\documentclass[12pt]{amsart}
\usepackage{mathrsfs}
\usepackage{amsfonts,amsthm}
\usepackage{amssymb}
\usepackage{mathtools}
\usepackage{physics2}
\usephysicsmodule{ab}
\usepackage[papersize={8.2in,11in},textwidth=16.4cm,textheight=22cm,centering]{geometry}
\usepackage{enumerate}
\usepackage{enumitem}

\usepackage{xcolor}
\definecolor{cite}{rgb}{0.30,0.60,1.00}
\definecolor{url}{rgb}{1.00,0.10,0.80}
\definecolor{link}{rgb}{0.00,0.00,1.00}

\usepackage[colorlinks,linkcolor=link,urlcolor=url,citecolor=cite,pagebackref,breaklinks]{hyperref}
\hypersetup{
pdfstartpage=1,
pdfstartview=FitH}

\usepackage{graphicx}

\usepackage{color}

\usepackage[OT2,T1]{fontenc}

\DeclareSymbolFont{cyrletters}{OT2}{wncyr}{m}{n}
\DeclareMathSymbol{\Sha}{\mathalpha}{cyrletters}{"58}

\DeclareFontFamily{U}{mathx}{\hyphenchar\font45}
\DeclareFontShape{U}{mathx}{m}{n}{
      <5> <6> <7> <8> <9> <10>
      <10.95> <12> <14.4> <17.28> <20.74> <24.88>
      mathx10
      }{}
\DeclareSymbolFont{mathx}{U}{mathx}{m}{n}
\DeclareMathAccent{\widecheck}{\mathalpha}{mathx}{"71}

 \usepackage{caption} 
\numberwithin{equation}{section}

\allowdisplaybreaks

\theoremstyle{plain}

\newtheorem*{theorem*}{Theorem}
\newtheorem{theorem}{Theorem}[section]
\newtheorem*{theorem A}{Xi's Theorem}
\newtheorem{lemma}[theorem]{Lemma}

\newtheorem{definition}[theorem]{Definition}
\newtheorem{proposition}[theorem]{Proposition}

\newtheorem{corollary}[theorem]{Corollary}

\newtheorem*{claim*}{Claim}

\theoremstyle{definition}
\newtheorem{remark}[theorem]{Remark}

\newtheorem*{rmk}{Remark}

\newcommand{\kl}{\mathrm{Kl}}

\DeclareMathOperator{\Mod}{mod}

\renewcommand{\bmod}[1]{\,(\Mod{ #1})}

\begin{document}

\title{Distinction between hyper-Kloosterman sums and multiplicative functions}
\author{Yang Zhang}

\address{School of Mathematics and Statistics, Xi'an Jiaotong University, Xi'an 710049, P. R. China}
\address{Morningside Center of Mathematics, Academy of Mathematics and Systems Science, Chinese Academy of Sciences, Beijing 100190, P.R. China}

\email{femistein.yang@gmail.com, 4120107016@stu.xjtu.edu.cn}

\subjclass[2020]{11L05, 11N64, 11N37, 11T23}

\keywords{hyper-Kloosterman sum, multiplicative functions, equidistribution, Birch sum, Sali\'e sum}

\begin{abstract}
Let $\kl_n(a,b;m)$ be the hyper-Kloosterman sum. Fix integers $n\geqslant2,a\neq0$, $b\neq0$ and $k\geqslant2$. For any $0\neq\eta\in\mathbb{C}$ and multiplicative function $f: \mathbb{N} \rightarrow \mathbb{C}$, we prove that  
$\kl_n(a,b;m)\neq\eta f(m)$ holds for $100\%$ square-free $k$-almost prime numbers $m$ and $100\%$ square-free numbers $m$. Counterintuitively, if $\kl_n(a,b;p)=\eta f(p)$ holds for all but finitely many primes $p$, we further show that
\begin{align*}
    \ab|\{m\leqslant X:\kl_n(a,b;m)=\eta f(m), m \text{ square-free }k\text{-almost prime}\}|= O(X^{1-\frac{1}{k+1}}).
\end{align*} 
These results overturn the general belief that $\kl_n(a,b;m)$ is nearly multiplicative in $m$, and that its distribution at almost prime moduli $m$ closely approximates that at primes.

Moreover, we prove that these results also hold for general algebraic exponential sums satisfying some natural conditions.

\end{abstract}

\dedicatory{{\it \small Dedicated to Professor Fei Xu}}

\maketitle


\section{Introduction}

The hyper-Kloosterman sum
\begin{align}\label{eq:Kloosterman}
\kl_n(a,b;m)\coloneqq\sum_{\substack{x_1,\dots,x_n \bmod{m}\\x_1\cdots x_n\equiv1\bmod {m}}}e\Big(\frac{a(x_1+\cdots +x_{n-1})+bx_n}{m}\Big),
\end{align}
defined for all $n\geqslant2,a,b\in\mathbb{Z}$ and $m\in\mathbb{N}$, satisfies the following twisted multiplicativity:
\begin{align*}
\kl_n(a,b;m_1m_2)=\kl_n(a\overline{m_2},b\overline{m_2};m_1)\kl_n(a\overline{m_1},b\overline{m_1};m_2)
\end{align*}
for any $m_1,m_2\in\mathbb{N}$ with $(m_1,m_2)=1$. Here, $m_2\overline{m_2}\equiv1\bmod{m_1}$ and $m_1\overline{m_1}\equiv1\bmod{m_2}$. 

Given integers $n\geqslant2,a,b$, then $\kl_n(a,b;m)$ is multiplicative in $m$ if and only if $ab=0$. In view of the resemblance between multiplicativity and twisted multiplicativity, even when $ab\neq0$, $\kl_n(a,b;m)$ is generally considered to behave like a multiplicative function in $m$. It is therefore well believed that studying the information of $\kl_n(a,b;m)$ at almost-primes $m$ provides a good approximation for that at primes. 

Denote by $\mathcal{P}$ the set of primes. Deligne \cite{Deligne77} proved the Weil bound:
\begin{align*}
\ab|\kl_n(a,b;p)|\leqslant np^{\frac{n-1}{2}},
\end{align*}
for any $a,b\in\mathbb{Z}$ and $p\in\mathcal{P}$ with $p\nmid a$. 

Restricting to the case $n=2$, one has $\kl_2(a,b;p)/\sqrt{p}\in[-2,2]$. Fixing $a\neq0$ and $b=1$, Katz \cite[Chapter1]{katz80} proposed the following three problems (also discussed in \cite{Xi20}), the last two problems here being equivalent transformations of Katz's original formulation.
\begin{enumerate}[label=(\Roman*)]
\item Does the density of $\{p\in\mathcal{P}:\kl_2(a,1;p)>0\}$ in $\mathcal{P}$ exist? If yes, is it equal to $\frac{1}{2}$?
\item Is there a measure on $[-2,2]$ such that $\{\frac{\kl_2(a,1;p)}{\sqrt{p}}:p\in\mathcal{P}\}$ has equidistribution?
\item Is there a Maass form $F$ of level $q$ with $q$ being a power of 2 such that
\begin{align*}
\lambda_F(p)=\frac{\kl_2(a,1;p)}{\sqrt{p}}
\end{align*}
holds for all but finitely many $p\in\mathcal{P}$? Here, $\lambda_F(p)$ is the $p$-th Fourier coefficient of $F$.
\end{enumerate}

To the best of the knowledge of the author, almost all research done on these problems to this day are concentrated on their corresponding almost-prime versions. 

For Problem I, Fouvry and Michel \cite{Fouvry-Michel07} proved that there are at least $\gg$ $X/\log X$ square-free numbers $m$ with at most $23$ prime factors in $[X,2X]$ such that $\kl_2(a,1;m)>0$. Subsequently, the constant $23$ was sharpened in \cite{Sivak-Fischler09, Matomaki11, Xi15, Xi18} and \cite{Drappeau-Maynar19}. 

For Problem II, Michel \cite{Michel95} proved that there is a positive proportion of pairs of primes $p,q\leqslant X$ such that $\ab|\frac{\kl_2(a,1;pq)}{4\sqrt{pq}}|\geqslant0.16$, and for any $\epsilon>0$, a positive proportion such that $\ab|\frac{\kl_2(a,1;pq)}{4\sqrt{pq}}|\leqslant\epsilon$. 

Denote by $\mu$ the M\"obius function, and $\omega(m)$ counts the number of distinct prime factors of $m$. 
For an integer $k\geqslant1$, the set of square-free $k$-almost primes defined by
\begin{align}
\Pi_k\coloneqq\{m\in\mathbb{N}:\mu^2(m)=1,~\omega(m)=k\},
\end{align}
and for all $X\geqslant1$, denote
\begin{align}
\Pi_k(X)\coloneqq\Pi_k\cap[1,X].
\end{align}
In order to give a negative answer to Problem III, Xi \cite{Xi20} proved the following theorem (main results of \cite[Theorem~1.1, Theorem~1.2]{Xi20}).
\begin{theorem A}
For any $0\neq\eta\in\mathbb{R}$ and each primitive Hecke-Maass cusp form $F$ of trivial nebentypus, there exists $r=r(\eta)\geqslant2$ such that 
\[
    \ab|\{m\in\bigcup_{k\leqslant r}\Pi_k(X):\kl_2(1,1;m)\neq\eta \sqrt{m}\lambda_F(m)\}|\gg\frac{X}{\log X}
\] 
holds for large $X$. In particular, one may take $r(\pm1)=100$.
Here, $\lambda_F(n)$ is the $n$-th Fourier coefficient of $F$ and equals the eigenvalue of the $n$-th Hecke operator.
\end{theorem A}
Note that the $\sqrt{m}\lambda_F(m)$ in Xi's Theorem is a multiplicative function of $m\in\mathbb{N}$. Considering the general belief that $\kl_2(a,1;m)$ are close to being multiplicative for $m$, it was then asserted in \cite{Xi20} that Xi's Theorem provides a partial negative answer to Problem III.

All these results on the above Katz's problems lead us to investigate and clarify the relationship between hyper-Kloosterman sums and multiplicative functions. Surprisingly, we find that $\kl_n(a,b;m)$ with fixed $n\geqslant2,a\neq0,b\neq0$, differs fundamentally and significantly from any complex valued multiplicative function of $m$, in a sense to be made more clear by the following three main theorems of this paper.

\begin{theorem}\label{thm:distinctness}
Fix integers $n\geqslant2,a\neq0,b\neq0$ and $k\geqslant2$. For any $0\neq\eta\in\mathbb{C}$ and multiplicative function $f: \mathbb{N} \rightarrow \mathbb{C}$, we have
\begin{align}\label{k-almost}
     \ab|\{m\in\Pi_k(X):\kl_n(a,b;m)=\eta f(m)\}|\leqslant \pi\ab(\frac{X}{L_k})+5\ab(\ab|ab|+\sqrt{nk})X^{1-\frac{1}{nk}}
\end{align}
for all $X\geqslant1$ and
\begin{align}\label{k-almost1}
     \ab|\{m\in\Pi_k(X):\kl_n(a,b;m)\neq\eta f(m)\}|= \ab|\Pi_k(X)|+O\ab(\pi\ab(\frac{X}{L_k}))
\end{align}
as $X\rightarrow +\infty$.

In particular, \eqref{k-almost1} implies that $\kl_n(a,b;m)\neq\eta f(m)$ holds for $100\%$ square-free $k$-almost primes $m$.
Here, $L_k$ denotes the product of the first $(k-1)$ primes, $\pi(\cdot)$ is the prime counting function.
\end{theorem}

\begin{rmk}

\hspace*{\fill}
\begin{enumerate}[leftmargin=*,align=left]
\item The inequality \eqref{k-almost} in the above theorem is sharp, so the error term of the asymptotic formula \eqref{k-almost1} is optimal. Indeed, let $f$ be a multiplicative function satisfying $f(p)=1$ for all $p\mid L_k$ and $f(p)=\eta^{-1}\kl_n(a,b;pL_k)$ for all $p$ with $(p,L_k)=1$.
Then we know that
\begin{align*}
\kl_n(a,b;pL_k)=\eta f(pL_k)
\end{align*}
for all primes $p$ coprime to $L_k$. Hence,
\begin{align*}
\ab|\{m\in\Pi_k(X):\kl_n(a,b;m)=\eta f(m)\}|\geqslant \pi\ab(\frac{X}{L_k})-k+1.
\end{align*}

\item Since the subset of those $m\in \Pi_k$ with $\kl_n(a,b;m)\neq\eta f(m)$ has density one in $\Pi_k$, our result demonstrates in particular that the distribution of such natural numbers $m\in\Pi_k$ is independent of that of the set of primes $p$ with $\kl_n(a,b;p)\neq\eta f(p)$ for all $k\geqslant2$.

 \item Applying Theorem~\ref{thm:distinctness} to the multiplicative function $\sqrt{m}\lambda_F(m)$, we know that the set $\{m\in\Pi_k:\kl_2(1,1;m)\neq\pm\sqrt{m}\lambda_F(m)\}$ actually always has density one in $\Pi_k$ for all $k\geqslant2$, a fact which is immune to the change of the set $\{p\in\mathcal{P}:\kl_2(1,1;p)\neq\pm\sqrt{p}\lambda_F(p)\}$ as the Hecke-Maass form $F$ varies. In view of this result, it now seems far-fetched to claim that Xi's Theorem is logically related to Problem III at all.
\end{enumerate}

\end{rmk}

\begin{theorem}\label{thm:square-free}
Fix integers $n\geqslant2,a\neq0,b\neq0$, then for any complex numbers $\eta\neq0$ and complex valued multiplicative function $f$, we have
\[
    \ab|\{m\leqslant X:\kl_n(a,b;m)=\eta f(m),\mu^{2}(m)=1\}|\leqslant \pi\ab(X)+\ab(4\ab|ab|n^2+\beta)Xe^{-\frac{2\sqrt{\log X}}{n}}
\] 
for all $X\geqslant1$, and
\[
    \ab|\{m\leqslant X:\kl_n(a,b;m)\neq\eta f(m),\mu^{2}(m)=1\}|=\frac{X}{\zeta\ab(2)}+O\ab(\frac{X}{\log X})
\] 
as $X\rightarrow +\infty$.

In particular, this asymptotic formula implies that $\kl_n(a,b;m)\neq\eta f(m)$ holds for $100\%$ square-free numbers $m$. 
Here, $\beta$ is an absolute constant and $\zeta(2)=\sum_{m=1}^{\infty}\frac{1}{m^2}$.
\end{theorem}

\begin{rmk}
The inequality in the above theorem is sharp. Let $f$ be a multiplicative function with $f(p)=\eta^{-1}\kl_n(a,b;p)$ for all primes $p$.
It is then easy to see that
\[
    \ab|\{m\leqslant X:\kl_n(a,b;m)=\eta f(m),\mu^{2}(m)=1\}|\geqslant \pi\ab(X).
\] 
\end{rmk}

One sees immediately that all the results in Theorem \ref{thm:distinctness} and \ref{thm:square-free} are uniform in $\eta$ and $f$. The main reason is that our proofs are intrinsic for $\eta$ and $f$ since we only use non-vanishing of $\eta$ and multiplicativity of $f$. 

Counterintuitively, when $\kl_n(a,b;p)=\eta f(p)$ holds for all but finitely many primes $p$, for $k\geqslant2$, the following theorem shows that the number of square-free $k$-almost primes $m$ with $\kl_n(a,b;m)=\eta f(m)$ is significantly smaller than the upper bound $\pi(\frac{X}{L_k})$ in Theorem \ref{thm:distinctness}.
\begin{theorem}\label{thm:repal}
Let $\eta\neq0$ be a complex number.
Fix integers $n\geqslant2,$ $a\neq0,b\neq0$ and $k\geqslant2$. Let $f: \mathbb{N} \rightarrow \mathbb{C}$ be a multiplicative function satisfying 
$\kl_n(a,b;p)=\eta f(p)$ for all but finitely many prime $p$. Then
\begin{align*}
\ab|\{m\in\Pi_k(X):\kl_n(a,b;m)=\eta f(m)\}|\leqslant\ab(\ab|ab|+nk)X^{1-\frac{1}{k+1}}+\binom{p_f}{k}
\end{align*}
holds for all integers $k\geqslant2$ and all $X\geqslant1$. 

Here, $p_f$ is the maximal prime $p$ with $\kl_n(a,b;p)\neq\eta f(p)$, $\binom{p_f}{k}$ is the $k$-th binomial coefficient.
\end{theorem}

\begin{remark}
From the above three theorems (with $n\geqslant2,a\neq0,b\neq0$ fixed), we see that the twisted multiplicativity of $\kl_n(a,b;m)$ in $m$, 
despite closely resembling multiplicativity, is in reality fundamentally different from it.
Therefore, the strategy of approximating the distribution of $\kl_n(a,b;m)$ at prime moduli $m$ by studying its distribution at almost prime moduli is infeasible.    
\end{remark}

\hspace*{\fill}

\subsection*{Outline of the proof}
Our method differs from the arguments with sieve methods, $\ell$-adic cohomology and spectral theory of automorphic forms in \cite{Fouvry-Micheal03,Fouvry-Michel07,Xi15,Xi18,Xi20}. The proof here is built on algebraic properties of hyper-Kloosterman sums and elementary extremal combinatorics. In particular, our main tools are the following two lemmas, which as far as we know is totally new for studying hyper-Kloosterman sums.

\begin{lemma}[Lemma~\ref{lm:irational}]\label{lm:1}
Let $p$ be a prime, $a,b,t$ three integers with $(abt,p)=1$. For any integer $n\geqslant2$, if
\begin{align*}
\frac{\kl_n(a,b;p)}{\kl_n(at,bt;p)}\in\mathbb{Q}
\end{align*}
then $t^n\equiv1\bmod p$.
\end{lemma}

\begin{lemma}[Lemma~\ref{lm:intersection}]\label{lm:2}
Let $S$ be a finite set with $N$ elements, $S_{1},...,S_{m}$ be $m$ subsets of $S$ satisfying the condition that the intersection of any $t$ of them contains at most 1 element, i.e., for any $t$ distinct integers $i_{1},\dots,i_{t}\in \{1,\dots,m\}$, we have
$\ab|S_{i_{1}}\cap\cdots\cap S_{i_{t}}|\leqslant 1$. Then
\begin{align*}
\mathop{\sum}_{i=1}^m\ab|S_{i}|\leqslant m+N\sqrt{m(t-1)}.
\end{align*}
\end{lemma}

 The basic strategies for proving all our theorems are similar. For the sake of discussion, we only give a sketch of the proof of Theorem \ref{thm:distinctness}. 

Given any integers $n\geqslant2,a\neq0,b\neq0$, complex number $\eta\neq0$ and complex valued multiplicative function $f$, denote
\begin{align}\label{definition0}
R_k(X)=\{m\in\Pi_k(X): \kl_n(a,b;m)=\eta f(m)\}
\end{align}
for all integers $k\geqslant2$ and real numbers $X\geqslant1$ (this is the special case of \eqref{de2} with $\eta_i=\eta$ for all $i$).
Partitioning the integers in $R_k(x)$ according to their largest prime factor, we have
\begin{align}\label{decomposition}
\ab|R_k(X)|=\sum_{p\in (1,\frac{X}{L_k}]}\ab|R^p_k(X)|,
\end{align}
where
\[
R^p_k(X)=\ab\{\frac{m}{p}: m\in R_k(X),P^+(m)=p\}
\]
for all primes $p$, and $P^+(m)$ denotes the largest prime factor of $m$. 

We will then estimate $\ab|R_k(X)|$ by dividing the summation interval in \eqref{decomposition} into four parts and bound them separately. As we mentioned above, the main tools we use are the algebraic properties of Klooserman sums and elementary combinatorics.

Using multiplicativity of $f$, twisted multiplicativity of $\kl_n(a,b;m)$ and Lemma \ref{lm:1}, for $X>(\ab|ab|)^{k+1}$, we prove the following three intersection properties (Lemma~\ref{lm:R^p-intersection}):
\begin{enumerate}[leftmargin=*,align=left]
\item $\ab|R^p_k(X)\cap R^q_k(X)|\leqslant1$ for any two different primes $p,q\in(X^{\frac{n}{n+1}},\frac{X}{L_k}]$;   
\item $\ab|\bigcap_{i=1}^{nk-n}R_k^{p_i}(X)|\leqslant1$ for any $(nk-n)$ different primes $p_1,\dots,p_{nk-n}\in(X^{\frac{1}{k}},X^{\frac{n}{n+1}}]$;
\item $\ab|\bigcap_{i=1}^{nk}R_k^{p_i}(X)|\leqslant1$ for any $nk$ different primes $p_1,\dots,p_{nk}\in(X^{\frac{1}{k+1}},X^{\frac{1}{k}}]$.
\end{enumerate}

For all prime  $p\in(X_1,X_2]\subseteq(1,X]$, by the definition of $R_k^p(X)$ we know that
\begin{align}\label{subseteq0}
R_k^p(X)\subseteq \{m\in\Pi_{k-1}\ab(\frac{X}{X_1}): P^+(m)\leqslant X_2\}
\end{align}
and denote~\eqref{de0} 
\[
\Pi_{k-1}\ab(\frac{X}{X_1},X_2)= \{m\in\Pi_{k-1}\ab(\frac{X}{X_1}): P^+(m)\leqslant X_2\}.
\]

Since we have intersection conditions (2) and (3),  by~\eqref{subseteq0} and Lemma~\ref{lm:2} we get the following two inequalities(Lemmas~\ref{lm:Sigma2},\ref{lm:Sigma3}):
\begin{align}\label{ineq1}
\sum_{p\in (X^{\frac{1}{k+1}},X^{\frac{1}{k}}]}\ab|R^p_k(X)| 
\leqslant c_1X^{1-\frac{1}{2k}}
\end{align}
and
\begin{align}\label{ineq2}
\sum_{p\in (X^{\frac{1}{k}},X^{\frac{n}{n+1}}]}\ab|R^p_k(X)|\leqslant c_2 X^{1-\frac{1}{nk}}
\end{align}
for all $X>(\ab|ab|)^{k+1}$. Here, $c_1,c_2$ are two constants determined by $n$ and $k$.

It is also easy to see that
\begin{align}\label{ineq3}
\sum_{p\in (1,X^{\frac{1}{k+1}}]}\ab|R^p_k(X)|\leqslant \ab|\{m\leqslant X:\omega(m)=k,P^+(m)\leqslant X^{\frac{1}{k+1}}\}|
\leqslant X^{\frac{k}{k+1}}.
\end{align}
The above three inequalities~\eqref{ineq1},~\eqref{ineq2} and~\eqref{ineq3} give the error term of the upper bound of $|R_k^p(X)|$. And the main term is given by the following inequality:
\begin{align*}
\sum_{p\in (X^{\frac{n}{n+1}},\frac{X}{L_k}]}\ab|R^p_k(X)|&=\mathop{\sum}_{\substack{p\in (X^{\frac{n}{n+1}},\frac{X}{L_k}]\\|R^p_{k}(X)|\leqslant1}}\ab|R^p_{k}(X)| + \mathop{\sum}_{\substack{p\in (X^{\frac{n}{n+1}},\frac{X}{L_k}]\\|R^p_{k}(X)|\geqslant2}}\ab|R^p_{k}(X)|\\
&\leqslant \pi\ab(\frac{X}{L_k})+ \mathop{\sum}_{\substack{p\in (X^{\frac{n}{n+1}},\frac{X}{L_k}]\\|R^p_{k}(X)|\geqslant2}}\ab|R^p_{k}(X)|
\end{align*}
From~\eqref{subseteq0} we know that $R_k^p(X)\subseteq\Pi_{k-1}(X^{\frac{1}{n+1}},\frac{X}{L_k})$ for all prime 
$p\in (X^{\frac{n}{n+1}},\frac{X}{L_k}]$. For $X>(\ab|ab|)^{k+1}$ by intersection condition (1) and the pigeonhole principle, 
we have (Lemma~\ref{lm:Sigma4})
\begin{align*}
\ab|\{p\in (X^{\frac{n}{n+1}},\frac{X}{L_k}]:|R^p_{k}(X)|\geqslant2\}|\leqslant\binom{\ab|\Pi_{k-1}(X^{\frac{1}{n+1}},\frac{X}{L_k})|}{2}
\leqslant X^{\frac{2}{n+1}}.
\end{align*}
Applying Lemma~\eqref{lm:2}, we derive (Lemma~\ref{lm:Sigma4})
\begin{align*}
 \mathop{\sum}_{\substack{p\in (X^{\frac{n}{n+1}},\frac{X}{L_k}]\\|R^p_{k}(X)|\geqslant2}}\ab|R^p_{k}(X)|&\leqslant X^{\frac{2}{n+1}}+ \ab|\Pi_{k-1}(X^{\frac{1}{n+1}},\frac{X}{L_k})|\sqrt{X^{\frac{2}{n+1}}(2-1)}\\
 &\leqslant2X^{\frac{2}{n+1}} 
\end{align*}
for $X>(\ab|ab|)^{k+1}$. 

Hence, by~\eqref{definition0} and~\eqref{decomposition} we conclude that
\[
\{m\in\Pi_k(X): \kl_n(a,b;m)=\eta f(m)\}\leqslant\pi\ab(\frac{X}{L_k})+O\ab(X^{1-\frac{1}{nk}})
\]
for $X\rightarrow+\infty$.

Besides the hyper-Kloosterman sum, our method is also able to treat general algebraic exponential sums~
(Theorem~\ref{thm:exp sum}, Remark~\ref{re:exp sum}).
For instance, we study the Birch sum (Theorem~\ref{thm:Birch sum}):
\begin{align}\label{Birch sum}
B(a,b;m)\coloneqq \sum_{x \bmod{m}}e\Big(\frac{ax^3+bx}{m}\Big)
\end{align} 
and the Sali\'e sum (Proposition~\ref{pro:Salie sum good}):
\begin{align}\label{Salie sum}
\widetilde{S}(a,b;m)\coloneqq\sum_{\substack{x \bmod{m}\\x\overline{x}\equiv1\bmod {m}}}\left(\frac{x}{m}\right)e\Big(\frac{ax+b\overline{x}}{m}\Big).
\end{align}

\begin{remark}
In fact, inspired by the proof of Theorems~\ref{thm:distinctness},~\ref{thm:square-free} and~\ref{thm:repal}, we introduce a new concept (Definition~\ref{def:Kloostermanian}): \emph{Kloostermanian function}. And we generalize these three theorems of hyper-Kloosterman sums to
Kloostermanian function (Theorme~\ref{thm:Kloostermanian}). 
\end{remark}

The paper is organized as follows. We first record some basic algebraic properties of $\kl_n(a,b;m)$, especially we prove Lemma~\ref{lm:1} in \S~\ref{sub:algebraic}. Then in \S~\ref{sub:definition} we give some definitions, for example those of $R_k$ and $R_k^p(X)$, which will be fundamental in our proofs. After that, we obtain two intersection properties Lemmas~\ref{lm:R^p-intersection} and ~\ref{lm:R^p-intersection0}
using the previous algebraic properties of $\kl_n(a,b;m)$. These two lemmas together with Lemma~\ref{lm:2} (Lemma~\ref{lm:intersection}) pave the way for our proofs of main theorems in subsequent sections. In \S~\ref{Sec.3}, we give the proof of Theorem~\ref{thm:distinctness}. In \S\S~\ref{Sec.4},\ref{Sec.5}, we prove Theorem~\ref{thm:square-free1} and Theorem~\ref{thm:repal1}, which include Theorem~\ref{thm:square-free} and Theorem~\ref{thm:repal1} as special cases respectively. In the Appendix, we introduce the concept of Kloostermanian functions and more generally almost-Kloostermanian functions, and extend our results for hyper-Kloosterman sums to their corresponding versions as promised. In particular,
a algebraic sums with some natural properties is Kloostermanian function. As example, the Birch sum and the Sali\'e sum are studied.

\subsection*{Notation and Convention}

In this paper, $\ab|\cdot|$ means the cardinality for a finite set and the absolute value for a real number, respectively. It will be clear from the context which meaning is taken. For quantities $\mathcal{A}$ and $\mathcal{B}$, if there is an explicit constant $c$ such that $\ab|\mathcal{A}|\leqslant c \ab|\mathcal{B}|$, we denote $\mathcal{A}=O\ab(\mathcal{B})$ and do not need $c$ to be an absolute constant like usual convention.
\begin{itemize}[leftmargin=*,align=left]
\item $e(x)\coloneqq e^{2\pi ix}$, $\zeta_m\coloneqq e(\frac{1}{m})$ and $\pi(\cdot)$ is the prime counting function. 
\item $ L_k$ is the product of the first $(k-1)$ primes for $k\geqslant2$.
\item $P^+(1)=1$ and $P^+(m)$ denote the largest prime factor of $m$ for an integer $m\geqslant2$.
\item For any integers $u,v$, $(u,v)=1$ means $u$ coprime to $v$; if $v\geqslant1$, $\left(\frac{u}{v}\right)$ is Jacobi symbol.
\end{itemize}
When a modulus $m$ is clear from the context, denote by $\overline{a}$ the multiplicative inverse of $a\bmod m$ for an integer $a$ coprime to $m$, i.e., $a\overline{a}\equiv1\bmod m.$

\subsection*{Acknowledgements}
 I am deeply indebted to Jing Liu for spending an enormous amount of time checking the details of the proofs and revising the writing structure and grammar of this paper. I would like to thank Professor Ye Tian for inviting me to visit MCM Chinese Academy of Sciences and Professor Zhizhong Huang for warm encouragement and valuable suggestions. Earlier versions of this paper primarily deal with Kloosterman sums. I am very grateful to my advisor professor Ping Xi for encouraging me to consider hyper-Kloosterman sums, and for providing many suggestions to improve the structure of the paper.
 
\smallskip

\section{Preliminaries}\label{Sec.2}
In this section, we first collect several algebraic properties of hyper-Kloosterman sums, then use them to derive some preparatory lemmas for proving our main theorems.

\subsection{Algebraic properties of hyper-Kloosterman sums}\label{sub:algebraic}

\begin{lemma}[{\cite[(2.4)]{Wan95}}]\label{lm:congruence}
Fix integer $n\geqslant2$, for any prime $p$ and integers $a,b$ we have
\begin{align*}
\kl_n(a,b;p)\equiv(-1)^{n-1}\bmod{(\zeta_{p}-1)}.
\end{align*}
\end{lemma}

\begin{lemma}\label{lm:twistmuliplicative}
Fix integer $n\geqslant2$. Then for any integers $a,b$ and $u,v\geqslant1$ with $(u,v)=1$, we have
\begin{align*}
\kl_n(a,b;uv)=\kl_n(a\overline v,b\overline v;u)\kl_n(a\overline u,b\overline u;v).
\end{align*}
\end{lemma}
\proof
This follows directly from the Chinese Remainder Theorem.
\endproof
The following corollary is an immediate consequence of the previous lemmas.
\begin{corollary}\label{cor:nonzero}
For any integers $a,b$ and square-free positive integer $m$, $\kl_n(a,b;m)\neq0$.
\end{corollary}

\begin{lemma}\label{lm:irational}
Let $p$ be a prime, $a,b$ and $t$ integers with $(p,abt)=1$. For any integer $n\geqslant2$, if
\begin{align*}
\frac{\kl_n(a,b;p)}{\kl_n(at,bt;p)}\in\mathbb{Q}
\end{align*}
then $t^n\equiv1\bmod{p}$. 
\end{lemma}
\proof
When $p=2$, it is obvious.

For $p> 2$, assume that
\begin{align*}
\frac{\kl_n(a,b;p)}{\kl_n(at,bt;p)}\in\mathbb{Q}.
\end{align*}
From $(p,t)=1$ we know that $\kl_n(at,bt;p)$ is a Galois conjugate of $\kl_n(a,b;p)$. Hence we get
\begin{align*}
\frac{\kl_n(a,b;p)}{\kl_n(at,bt;p)}=\pm1
\end{align*}
We also have 
\begin{align*}
\kl_n(a,b;p)+\kl_n(at,bt;p)\equiv2(-1)^{n-1}\not\equiv0\bmod{(\zeta_p-1)}.
\end{align*}
by Lemma~\ref{lm:congruence}. Therefore,
\begin{align*}
\kl_n(a,b;p)=\kl_n(at,bt;p).
\end{align*}
Since $(p,ab)=1$, from \cite[Theorem 1.1]{Wan95} we also know that 
\[
\ab[\mathbb{Q}(\zeta_p): \mathbb{Q}\ab(\kl_n(a,b;p))]\mid n.
\]
Hence, we conclude that $t^n\equiv1\bmod{p}.$
\endproof

\begin{lemma}[{\cite[Proposition 2.4]{Washington97}}]\label{lm:cyclotomic}
Let $u,v$ be two positive integers with $(u,v)=1$. Then $ \mathbb{Q}(\zeta_{u})\cap\mathbb{Q}(\zeta_{v})=\mathbb{Q}$.
\end{lemma}

\subsection{Definitions and some preparatory lemmas}\label{sub:definition} 
For an integer $k\geqslant1$ and a real number $X\geqslant1$, denote
\[
\pi_{k}\ab(X)\coloneqq\ab|\Pi_{k}\ab(X)|,
\]
which admits the following asymptotic formula
\cite[Chapter 10, Theorem 10.3]{De-Luca12}: 
\begin{align}\label{asymptotic}
\pi_k\ab(X)\sim \frac{X(\log\log X)^{k-1}}{(k-1)!\log X}
\end{align}
as $X\rightarrow+\infty$. And we have the following inequality for large $k$.
\begin{lemma}\label{lm:k-prime bound}
 There exists a absolute positive constant $C$ such that  
 \begin{align*}
\pi_k\ab(X)\leqslant C\frac{X}{\log X}e^{-\sqrt{\log X}}
\end{align*} 
 for any $X\geqslant3$ and any integers $k\geqslant\frac{\sqrt{\log X}}{3}$.
\end{lemma}
\proof 
It is a direct corollary of the inequality $k!\geqslant\ab(\frac{k}{e})^k$
and the Hardy-Ramanujan inequality \cite{Hardy-Ramanujan17}:
\[
\pi_k\ab(X)\leqslant C_1\frac{X(\log\log X +C_2)^{k-1}}{(k-1)!\log X}.
\]
Here, $C_1$ and $C_2$ are two absolute constants.
\endproof

We also denote
\begin{align}\label{de0}
\Pi_{k}\Big(X,y\Big)\coloneqq\{m\in\Pi_{k}\ab(X):P^+(m)\leqslant y\}.
\end{align}

Fix integers $n\geqslant2$ and $a\neq 0\neq b$, for any given sequence of non-zero complex numbers $(\eta_i)_{i=1}^{\infty}$ and complex valued multiplicative function $f$, denote
\begin{align}\label{de1}
R_{k}\coloneqq\{m\in \Pi_k:\kl_n(a,b;m)=\eta_k f(m)\};
\end{align}
\begin{align}\label{de2}
R_{k}(X)\coloneqq \{m\in \Pi_k(X):\kl_n(a,b;m)=\eta_k f(m)\};
\end{align}
\begin{align}\label{de3}
r_{k}(X)\coloneqq\ab|R_{k}(X)|.
\end{align}

In order to obtain an optimal upper bound for its cardinality $r_k(X)$, we will partition $R_k(X)$ according to the largest prime factor of its elements, which leads to the introduction of the following notation:
\begin{align}\label{de4}
R^p_{k}(X)\coloneqq\ab\{\tfrac{m}{p}: m\in R_k(X),\, P^+(m)=p\}.
\end{align}
These subsets $R_k^p(X)$ of $R_k(X)$ will be the main player in our estimations throughout this paper.

Let us remark here that, whereas the data of the sequence $(\eta_i)$ and the function $f$ is certainly part of the definition of the above objects, but since no confusion will be caused in this work, we omit them from the notations for simplicity.

For $k\geqslant2$, the following observations are immediate once we unravel the relevant definitions.

\begin{align}\label{lm:sum}
r_{k}(X)&=\mathop{\sum}_{p\in(1,X]}\ab|R^p_{k}(X)|\\ 
& =\mathop{\sum}_{p\in \left(1,\frac{X}{L_k}\right]}\ab|R^p_{k}(X)|, \nonumber
\end{align}

\begin{align}\label{lm:p-sum}
R^p_{k}(X)\subseteq\Pi_{k-1}\Big(\frac{X}{p},p\Big).
\end{align}
Despite their innocent appearance, these observations will nevertheless play a central role in the proof of our main theorems.

Denote 
\begin{align}
R^{p}(X)\coloneqq\bigcup_{k=2}^{\infty}R^p_k(X).
\end{align}

\begin{lemma}\label{lm:divisor}
Let $p,q$ be two distinct primes with $(pq,ab)=1$, and $u,v\in R^p(X)\cap R^q(X)$ with $u\neq v$. Then we have $pq\mid(u^n-v^n)$.

More generally, if $p_{1},\dots,p_{t}$ are $t$ distinct primes with $(\prod_{i=1}^{t}p_{i},ab)=1$ and $u,v$ are distinct elements in $\bigcap_{i=1}^{t}R^{p_i}(X)$ for any $t\geqslant2$, then we have 
$(\prod_{i=1}^{t}p_{i})\mid(u^n-v^n)$.
\end{lemma}
\proof
From $u\not=v\in R^p(X)\cap R^q(X)$ we know that there exist integers $k_1,k_2\in[2,+\infty)$ such that
\begin{align*}
u\in R_{k_1}^p(X)\cap R_{k_1}^q(X),\, v\in R_{k_2}^p(X)\cap R_{k_2}^q(X).
\end{align*}
By the definition \eqref{de4}, we have
\begin{align*}
\kl_n(a,b;pu)=\eta_{k_1} f(pu),\, \kl_n(a,b;qu)=\eta_{k_1} f(qu)
\end{align*}
and
\begin{align*}
\kl_n(a,b;pv)=\eta_{k_2} f(pv),\, \kl_n(a,b;qv)=\eta_{k_2} f(qv).
\end{align*}

Since $f$ is a multiplicative function and $\kl_n(a,b;m)\not=0$ for all square-free positive integers $n$, we obtain
\begin{align*}
\frac{\kl_n(a,b;pu)}{\kl_n(a,b;pv)}\frac{\kl_n(a,b;qv)}{\kl_n(a,b;qu)}=1.
\end{align*}
By twisted multiplicativity of hyper-Kloosterman sums (Lemma~\ref{lm:twistmuliplicative}), we derive
\begin{align*}
\frac{\kl_n(a\overline u,b\overline u;p)}{\kl_n(a\overline v,b\overline v;p)}=\frac{\kl_n(a\overline p,b\overline p;v)}
{\kl_n(a\overline p,b\overline p;u)}\frac{\kl_n(a,b;qu)}{\kl_n(a,b;qv)}.
\end{align*}
Since $p$ is coprime to $u,v$ and $q$, we get
\begin{align*}
\frac{\kl_n(a\overline u,b\overline u;p)}{\kl_n(a\overline v,b\overline v;p)}\in \mathbb{Q}(\zeta_p)\cap\mathbb{Q}(\zeta_{quv})=\mathbb{Q}
\end{align*} from Lemma~\ref{lm:cyclotomic}.

By Lemma~\ref{lm:irational}, we know $(u\overline v)^n\equiv1\bmod p$, namely $p\mid(u^n-v^n)$. Similarly, we have $q\mid(u^n-v^n)$, and hence $pq|(u^n-v^n)$. This proves the first statement, and the second one is a direct corollary.
\endproof

\begin{lemma}\label{lm:R^p-intersection}
Let $t\geqslant2$ be an integer, $p_{1},\dots,p_{t}\in(X^{\frac{n}{n+t}},X]$ be $t$ distinct primes and
$X>\ab(\ab|ab|)^{\frac{n+t}{n}}$. Then we have 
\begin{align*}
\ab|\bigcap_{i=1}^{t}R^{p_i}(X)|\leqslant1.
\end{align*}
This intersection property still holds when $R^{p_i}(X)$ is replaced by $R_k^{p_i}(X)$ for any $k\geqslant2$.
\end{lemma}
\proof

Assuming that there exist $u\neq v\in \bigcap_{i=1}^{t}R^{p_i}(X)$, by Lemma~\ref{lm:divisor}, we get 
\begin{align*}
(\prod_{i=1}^{t}p_{i})\mid(u^n-v^n).
\end{align*}
Hence, we obtain
\begin{align*}
X^{\frac{nt}{n+t}}< \prod_{i=1}^{t}p_{i}\leqslant \max\{u^2,v^2\}\leqslant\Big(\dfrac{X}{p_1}\Big)^n< X^{n(1-\frac{n}{n+t})},
\end{align*}
a contradiction.
\endproof

Recall that we have defined 
\begin{align}\label{R_1}
R_{1}=\{p\in \mathcal{P}:\kl_n(a,b;p)=\eta_1 f(p)\}.
\end{align}

\begin{lemma}\label{lm:divisor0}
Let $p,q\in R_1$ be two distinct primes with $(pq,ab)=1$, and $u\in R^p(X)\cap R^q(X)$. Then we have $pq\mid(u^n-1)$.

More generally, if $p_{1},\dots,p_{t}\in R_1$ are $t$ distinct primes with $(\prod_{i=1}^{t}p_{i},ab)=1$ and $u$ is an element in $\bigcap_{i=1}^{t}R^{p_i}(X)$ for any $t\geqslant2$, then we have 
$(\prod_{i=1}^{t}p_{i})\mid(u^n-1)$.
\end{lemma}
\proof
From $u\in R^p(X)\cap R^q(X)$, we know that $u\in R^p_{k}(X)\cap R^q_{k}(X)$ for some $k\geqslant2$. Then we get
\begin{align*}
\kl_n(a,b;pu)=\eta_k f(pu),\, \kl_n(a,b;qu)=\eta_k f(qu).
\end{align*}
Since $p,q\in R_1$ and $f$ is a multiplicative function, we know that
\begin{align*}
\kl_n(a,b;pu)=\eta_k f(u)\frac{\kl_n(a,b;p)}{\eta_1},\, \kl_n(a,b;qu)=\eta_k f(u)\frac{\kl_n(a,b;q)}{\eta_1}.
\end{align*}
Note that $\kl_n(a,b;m)\not=0$ for all square-free positive integers $m$, so we derive
\begin{align*}
\frac{\kl_n(a,b;pu)}{\kl_n(a,b;p)}\frac{\kl_n(a,b;q)}{\kl_n(a,b;qu)}=1.
\end{align*}
Hence, using the same argument as in Lemma~\ref{lm:divisor}, we obtain $pq\mid(u^n-1)$.
This proves the first statement, and the second one is a direct corollary.

\endproof

\begin{lemma}\label{lm:R^p-intersection0}
Let $t\geqslant2$ be an integer, $p_{1},\dots,p_{t}\in R_1\cap(X^{\frac{n}{n+t}},X]$ be $t$ distinct primes and
$X>\ab(\ab|ab|)^{\frac{n+t}{n}}$. Then we have 
\begin{align*}
\ab|\bigcap_{i=1}^{t}R^{p_i}(X)|=0.
\end{align*}
Again this intersection property is retained when $R^{p_i}(X)$ is replaced by $R^{p_i}_k(X)$ for any $k\geqslant2$.
\end{lemma}
\proof
Assume that there exists $u\in \bigcap_{i=1}^{t}R^{p_i}(X)$. Applying Lemma \ref{lm:divisor0}, we get 
\begin{align*}
\prod_{i=1}^{t}p_i\mid(u^n-1).
\end{align*}
We also know that $\omega(u)\geqslant1$ from $u\in \bigcap_{i=1}^{t}R^{p_i}(X)$.
Hence, $u^n-1>0$ and we obtain
\begin{align*}
X^{\frac{nt}{n+t}}< \prod_{i=1}^{t}p_i\leqslant u^n-1<u^n\leqslant\Big(\frac{X}{p_1}\Big)^n< X^{n(1-\frac{n}{n+t})},
\end{align*}
a contradiction.
\endproof

The following elementary result is well-known in extremal combinatorics. For lack of a convenient reference, we provide a proof here.

\begin{lemma}\label{lm:intersection}
 Let $S$ be a finite set with $N$ elements, $S_{1},...,S_{m}$ be $m$ subsets of $S$ satisfying the condition that the intersection of any $t$ of them contains at most 1 element, i.e., $\ab|S_{i_{1}}\cap\cdots\cap S_{i_{t}}|\leqslant 1$ holds for any $t$ distinct integers $i_{1},\dots,i_{t}\in \{1,\dots,m\}$. Then
\begin{align*}
\mathop{\sum}_{i=1}^m\ab|S_{i}|\leqslant m+N\sqrt{m(t-1)}.
\end{align*}
\end{lemma}
\proof
By the intersection condition, we know that any subset of $2$ elements of $S$ is a subset of at most $(t-1)$ distinct sets $S_{i}$.
Double counting on the collection of $2$-element subsets in all $S_i$ then gives
\begin{align*}
\mathop{\sum}_{i=1}^m\binom{\ab|S_{i}|}{2}\leqslant (t-1)\binom{|S|}{2}.
\end{align*}
Let $E$ denote the sum $\mathop{\sum}_{i=1}^m\ab|S_{i}|$, using the Cauchy-Schwartz inequality
\begin{align*}
(\mathop{\sum}_{i=1}^m\ab|S_{i}|)^2\leqslant m\mathop{\sum}_{i=1}^m\ab|S_{i}|^2
\end{align*}
we know that
\begin{align*}
\frac{E^2}{m}-E\leqslant (t-1)N(N-1).
\end{align*}
By Vieta's formula, we see that
\begin{align*}
E\leqslant \frac{m+\sqrt{m^2+4m(t-1)N(N-1)}}{2}
&\leqslant m+N\sqrt{m(t-1)}
\end{align*}
as desired.
\endproof

\section{Proof of Theorem \ref{thm:distinctness}}\label{Sec.3}
The goal in this section is to estimate
\[
r_k(X)=\ab|\{m\in\Pi_k(X): \kl_n(a,b;m)=\eta_k f(m)\}|
\]
for all integers $k\geqslant2$.

\proof[Proof of Theorem~\ref{thm:distinctness}]
Take $\eta_i=\eta$ for all $i$. When $1\leqslant X\leqslant(\ab|ab|)^{k+1}$, it is easy to see that
\begin{align*}
r_k(X)\leqslant X\leqslant\ab|ab|X^{1-\frac{1}{k+1}},
\end{align*}
which is evidently less than $\pi\ab(\frac{X}{L_k})+5(\ab|ab|+\sqrt{nk})X^{1-\frac{1}{nk}}$.

When $X>(\ab|ab|)^{k+1}$, recalling our observation \eqref{lm:sum}, we have
\begin{align*}
r_{k}\ab(X)&=\mathop{\sum}_{p\in(1,\frac{X}{L_{k}}]}\ab|R^p_{k}\ab(X)|\\
&\leqslant\Sigma^{1}_{k}+\Sigma^{2}_{k}+\Sigma^{3}_{k}+\Sigma^{4}_{k},
\end{align*}
where for all $X\geqslant1$ we have put
\begin{align*}
\Sigma^{1}_{k}&\coloneqq\mathop{\sum}_{p\in(1,X^{\frac{1}{k+1}}]}\ab|R^p_{k}(X)|;\\
\Sigma^{2}_{k}&\coloneqq\mathop{\sum}_{p\in(X^{\frac{1}{k+1}},X^{\frac{1}{k}}]}\ab|R^p_{k}(X)|;\\
\Sigma^{3}_{k}&\coloneqq\mathop{\sum}_{p\in(X^{\frac{1}{k}},X^{\frac{n}{n+1}}]}\ab|R^p_{k}(X)|;\\
\Sigma^{4}_{k}&\coloneqq\mathop{\sum}_{p\in(X^{\frac{n}{n+1}},\frac{X}{L_k}]}\ab|R^p_{k}(X)|.
\end{align*}
We are going to treat these four parts separately in Lemma \ref{lm:Sigma1}, \ref{lm:Sigma2}, \ref{lm:Sigma3} and \ref{lm:Sigma4}. Granting these results, we see easily that
\begin{align*}
r_k(X)&\leqslant\Sigma_k^1+\Sigma_k^2+\Sigma_k^3+\Sigma_k^4\\
&\leqslant X^{\frac{k}{k+1}}+\ab(1+\sqrt{nk})X^{1-\frac{1}{2k}}+
\ab(X^{\frac{n}{n+1}}+4\sqrt{nk}X^{1-\frac{1}{2k}})+
\ab(\pi\ab(\frac{X}{L_k})+2X^{\frac{2}{n+1}})\\
&\leqslant \pi\ab(\frac{X}{L_k})+5\ab(1+\sqrt{nk})X^{1-\frac{1}{nk}}\\
&\leqslant \pi\ab(\frac{X}{L_k})+5\ab(\ab|ab|+\sqrt{nk})X^{1-\frac{1}{nk}}.
\end{align*}

Now by the Prime Number Theorem, we have
\[
    \ab|\{m\in\Pi_k(X):\kl_n(a,b;m)\neq\eta f(m)\}|=\pi_k\ab(X)+O(\pi\ab(\frac{X}{L_k})).
\] 

It is easy to see that
\[
\pi\ab(\frac{X}{L_k})=o(\pi_k\ab(X))
\] for any fixed $k\geqslant2$ by the asymptotic formula (\ref{asymptotic}).
Hence, we complete the proof of Theorem \ref{thm:distinctness}.
\endproof

In the rest of this section we estimate the four parts of $r_k(X)$ separately as promised.
\begin{lemma}\label{lm:Sigma1}
For all integer $k\geqslant2$ and all $X\geqslant1$, we have
\begin{align*}
\Sigma^{1}_{k}&\leqslant X^{\frac{k}{k+1}}.
\end{align*}    
\end{lemma}
\proof
By \eqref{lm:p-sum}, we obtain
\begin{align*}
\Sigma^{1}_{k} &=\mathop{\sum}_{p\in(1,X^{\frac{1}{k+1}}]}\ab|R^p_{k}(X)|\\
&\leqslant\mathop{\sum}_{p\in(1,X^{\frac{1}{k+1}}]}\ab|\Pi_{k-1}\Big(\frac{X}{p},p\Big)|\\
&\leqslant\mathop{\sum}_{p\in(1,X^{\frac{1}{k+1}}]}\ab|\Pi_{k-1}\Big(X,X^{\frac{1}{k+1}}\Big)|\\
&\leqslant\pi\ab(X^{\frac{1}{k+1}})\ab|\Pi_{k-1}\Big(X,X^{\frac{1}{k+1}}\Big)|\\
&\leqslant X^{\frac{k}{k+1}}.
\end{align*}
\endproof

\begin{lemma}\label{lm:Sigma2}
For any integer $k\geqslant2$, when $X>(\ab|ab|)^{k+1}$, we have
\begin{align*}
\Sigma^{2}_{k}&\leqslant(1+\sqrt{nk})X^{1-\frac{1}{2k}}.
\end{align*}    
\end{lemma}
\proof
Since $X>(\ab|ab|)^{k+1}$, from Lemma~\ref{lm:R^p-intersection} we know that
\begin{align*}
\ab|\bigcap_{i=1}^{nk}R^{p_i}_{k}(X)|\leqslant1
\end{align*}
for any $nk$ distinct primes $p_{1},...p_{nk}\in (X^{\frac{1}{k+1}},X]$.
By \eqref{lm:p-sum}, we get
\begin{align*}
R^p_{k}(X)\subseteq\Pi_{k-1}\Big(\frac{X}{p},p\Big)\subseteq \Pi_{k-1}\Big(X^{\frac{k}{k+1}},X^{\frac{1}{k}}\Big)
\end{align*}
for all $p\in (X^{\frac{1}{k+1}},X^{\frac{1}{k}}]$.
Denote $I_k\coloneqq\mathcal{P}\cap(X^{\frac{1}{k+1}},X^{\frac{1}{k}}]$. Applying Lemma \ref{lm:intersection}, we derive 
\begin{align*}
\Sigma^{2}_{k}=\sum_{p\in I_{k}}\ab|R^p_{k}(X)|&\leqslant \ab|I_k|+\ab|\Pi_{k-1}\Big(X^{\frac{k}{k+1}},X^{\frac{1}{k}}\Big)|\sqrt{\ab|I_k|(nk-1)}\\
&\leqslant X^{\frac{1}{k}}+\ab( X^{\frac{1}{k}})^{k-1}\sqrt{nkX^{\frac{1}{k}}}\\
&\leqslant\ab(1+\sqrt{nk})X^{1-\frac{1}{2k}}.
\end{align*}
\endproof

\begin{lemma}\label{lm:Sigma3}
For any integer $k\geqslant2$, when $X>(\ab|ab|)^{k}$, we have
\begin{align*}
\Sigma^{3}_{k}&\leqslant X^{\frac{n}{n+1}}+4\sqrt{nk}X^{1-\frac{1}{2k}}.
\end{align*}    
\end{lemma}
\proof
Let us denote $T_{k}\coloneqq\lfloor \frac{(\frac{n}{n+1}-\frac{1}{k})\log X}{\log2} \rfloor$, $J_{t}\coloneqq(2^tX^{\frac{1}{k}},2^{t+1}X^{\frac{1}{k}}]$ for $0\leqslant t\leqslant T_{k}-1$ and $J_{T_k}\coloneqq(2^{T_k}X^{\frac{1}{k}},X^{\frac{n}{n+1}}]$. Then we have the following dyadic decomposition.
\begin{align*}
\Sigma^{3}_{k}=\sum_{t=0}^{T_{k}}\sum_{p\in J_{t}}\ab|R^p_{k}(X)|.
\end{align*}
Since $X> (\ab|ab|)^{k}$, from Lemma~\ref{lm:R^p-intersection} we know that
\begin{align*}
\ab|\bigcap_{i=1}^{(nk-n)}R^{p_i}_{k}(X)|\leqslant1
\end{align*}
for any $(nk-n)$ distinct primes $p_{1},...,p_{(nk-n)}\in (X^{\frac{1}{k}},X]$, a fortiori if all the $p_i\,'s$ are in a single $J_t$ for some $0\leqslant t\leqslant T_{k}$.
By \eqref{lm:p-sum}, we obtain
\begin{align*}
R^p_{k}(X)\subseteq\Pi_{k-1}\Big(\frac{X}{p},p\Big)\subseteq \Pi_{k-1}\Big(2^{-t}X^{\frac{k-1}{k}},2^{t+1}X^{\frac{1}{k}}\Big)
\end{align*}
for any $p\in J_t$ and $0\leqslant t\leqslant T_{k}$.
Applying Lemma \ref{lm:intersection}, we derive 
\begin{align*}
\sum_{p\in J_{t}}\ab|R^p_{k}(X)|\leqslant\ab| J_t\cap\mathcal{P}|+\ab|\Pi_{k-1}\Big(2^{-t}X^{\frac{k-1}{k}},2^{t+1}X^{\frac{1}{k}}\Big)|\sqrt{\ab| J_t\cap\mathcal{P}|(nk-n-1)}.
\end{align*}
Hence
\begin{align*}
\Sigma^{3}_{k}&=\sum_{t=0}^{T_{k}}\sum_{p\in J_{t}}\ab|R^p_{k}(X)|\\
&\leqslant\sum_{t=0}^{T_{k}}\ab| J_t\cap\mathcal{P}|+\ab|\Pi_{k-1}\Big(2^{-t}X^{\frac{k-1}{k}},2^{t+1}X^{\frac{1}{k}}\Big)|\sqrt{\ab| J_t\cap\mathcal{P}|(nk-n-1)}\\
&\leqslant X^{\frac{n}{n+1}}+\sum_{t=0}^{T_{k}}2^{-t}X^{\frac{k-1}{k}}\sqrt{nk(2^{t+1}X^{\frac{1}{k}}-2^{t}X^{\frac{1}{k}})}\\
&\leqslant X^{\frac{n}{n+1}}+\sqrt{nk}X^{1-\frac{1}{2k}}\sum_{t=0}^{T_{k}}2^{\frac{-t}{2}}\\
&\leqslant X^{\frac{n}{n+1}}+4\sqrt{nk}X^{1-\frac{1}{2k}}.
\end{align*}
\endproof

\begin{lemma}\label{lm:Sigma4}
For any integer $k\geqslant2$, when $X>(\ab|ab|)^{\frac{n+2}{n}}$, we have
\begin{align*}
\Sigma^{4}_{k}&\leqslant\pi\ab(\frac{X}{L_k})+2X^{\frac{2}{n+1}}.
\end{align*}    
\end{lemma}
\proof
We decompose $\Sigma^{4}_{k}$ into the following two parts.
\begin{align*}
\Sigma^{4}_{k}=\mathop{\sum}_{\substack{p\in (X^{\frac{n}{n+1}},\frac{X}{L_k}]\\|R^p_{k}(X)|\leqslant1}}\ab|R^p_{k}(X)| + \mathop{\sum}_{\substack{p\in (X^{\frac{n}{n+1}},\frac{X}{L_k}]\\|R^p_{k}(X)|\geqslant2}}\ab|R^p_{k}(X)|
\end{align*}
Evidently,
\begin{align*}
\mathop{\sum}_{\substack{p\in (X^{\frac{n}{n+1}},\frac{X}{L_k}]\\|R^p_{k}(X)|\leqslant1}}\ab|R^p_{k}(X)| 
\leqslant \mathop{\sum}_{p\in (X^{\frac{n}{n+1}},\frac{X}{L_k}]}1\leqslant\pi\ab(\frac{X}{L_k}).
\end{align*}
By \eqref{lm:p-sum}, we know that 
\begin{align*}
R^p_{k}(X)\subseteq\Pi_{k-1}\Big(\frac{X}{p},p\Big)\subseteq \Pi_{k-1}\ab(X^{\frac{1}{n+1}})
\end{align*}
for all primes $p\in (X^{\frac{n}{n+1}},\frac{X}{L_k}]$.

Since $X>(\ab|ab|)^{\frac{n+2}{n}}$, from Lemma~\ref{lm:R^p-intersection} we obtain  
\begin{align}\label{2-intersection}
\ab|R^{p}_{k}(X)\cap R^{q}_{k}(X)|\leqslant1,
\end{align}
for any two different primes $p,q\in (X^{\frac{n}{n+1}},\frac{X}{L_k}]$.

Denote
\[
U_k(X)\coloneqq\{p:p\in (X^{\frac{n}{n+1}},\frac{X}{L_k}],\ab|R^p_{k}(X)|\geqslant2\}.
\]  
Applying Lemma \ref{lm:intersection}, we derive
\begin{align*}
\mathop{\sum}_{\substack{p\in (X^{\frac{n}{n+1}},\frac{X}{L_k}]\\|R^p_{k}(X)|\geqslant2}}\ab|R^p_{k}(X)| 
&\leqslant \ab|U_k(X)|+\ab|\Pi_{k-1}\ab(X^{\frac{1}{n+1}})|\sqrt{\ab|U_k(X)|}\\
&=\ab|U_k(X)|+\pi_{k-1}\ab(X^{\frac{1}{n+1}})\sqrt{\ab|U_k(X)|}.
\end{align*}

According to (\ref{2-intersection}), we know that any 2-element subset of $\Pi_{k-1}\ab(X^{\frac{1}{n+1}})$ is contained in at most one
$R^p_{k}(X)$ for all primes $p\in U_k(X)$. By the definition of $U_k(X)$ and the pigeonhole principle, we find
\begin{align*}
\ab|U_k(X)|&\leqslant \binom{\ab|\Pi_{k-1}\ab(X^{\frac{1}{n+1}})|}{2}\\
&\leqslant \ab(\pi_{k-1}\ab(X^{\frac{1}{n+1}}))^2
\end{align*}
Then
\begin{align*}
\mathop{\sum}_{\substack{p\in (X^{\frac{n}{n+1}},\frac{X}{L_k}]\\|R^p_{k}(X)|\geqslant2}}|R^p_{k}(X)| 
&\leqslant \ab(\pi_{k-1}\ab(X^{\frac{1}{n+1}}))^2+\pi_{k-1}\ab(X^{\frac{1}{n+1}})\sqrt{\pi_{k-1}\ab(X^{\frac{1}{n+1}})^2}\\
&\leqslant2X^{\frac{2}{n+1}}.
\end{align*}
Hence
\begin{align*}
\Sigma^{4}_{k}\leqslant\pi\ab(\frac{X}{L_k})+2X^{\frac{2}{n+1}}.
\end{align*}
This completes the proof
\endproof

\section{Proof of Theorem \ref{thm:square-free}}\label{Sec.4}
To prove Theorem~\ref{thm:square-free}, it suffices to estimate the following quantity:
\begin{align}\label{r(X)'}
r\ab(X)\coloneqq\ab|\{m\leqslant X:\kl_n(a,b;m)=\eta_i f(m),\mu^{2}(m)=1\text{ if }\omega(m)=i\geqslant2\}|.
\end{align}

We denote
\begin{align}\label{P^0}
\mathcal{P}^{0}\coloneqq\mathcal{P}\backslash R_1=\{p\in\mathcal{P}:\kl_n(a,b;p)\neq\eta_1f(p)\};
\end{align}

\begin{align}\label{pi^0}
\pi^0(X)\coloneqq\ab|\{p\leqslant X:p\in\mathcal{P}^0\}|.
\end{align}

By the definition of $r_k(X)$ and \eqref{lm:sum}, we have that
\begin{align}\label{r(X)}
r\ab(X)=\sum_{k\in[2,\frac{\log X}{\log2}]}r_k(X)=\sum_{k\in[2,\frac{\log X}{\log2}]}\sum_{p\in(1,\frac{X}{L_k}]}\ab|R_k^p(X)|.
\end{align}

In order to give a sharp upper bound for $r\ab(X)$, we need to decompose it into the following four parts:
\begin{align}\label{sum1}
\Sigma^1\coloneqq\sum_{k\in[2,\frac{\sqrt{\log X}}{3}]}\sum_{p\in(1,X^{\frac{n}{n+1}}]}\ab|R_k^p(X)|;
\end{align}

\begin{align}\label{sum2}
\Sigma^2\coloneqq\sum_{k\in[2,\frac{\sqrt{\log X}}{3}]}\sum_{\substack{p\in\mathcal{P}^0\\ p\in(X^{\frac{n}{n+1}},\frac{X}{L_k}]}}\ab|R_k^p(X)|;
\end{align}

\begin{align}\label{sum3}
\Sigma^3\coloneqq\sum_{k\in[2,\frac{\sqrt{\log X}}{3}]}\sum_{\substack{p\in R_1\\ p\in(X^{\frac{n}{n+1}},\frac{X}{L_k}]}}\ab|R_k^p(X)|;
\end{align}

\begin{align}\label{sum4}
\Sigma^4\coloneqq\sum_{k\in(\frac{\sqrt{\log X}}{3},\frac{\log X}{\log2}]}r_k(X).
\end{align}

The desired estimates of these four sums will be given in Lemmas~\ref{pro:a},~\ref{pro:b},~\ref{pro:c} and~\ref{pro:d} separately. Using these estimates, we prove the following theorem, and Theorem~\ref{thm:square-free} will be the special case with $\eta_i=\eta$ for all $i$.

\begin{theorem}\label{thm:square-free1}
Fix integers $n\geqslant2,a\neq0,b\neq0$, then for all given sequence of nonzero complex numbers $(\eta_i)_{i=1}^{\infty}$ and multiplicative function $f$, 
\[
    \ab|\{m\leqslant X:\kl_n(a,b;m)=\eta_i f(m),\mu^{2}(m)=1\text{ if }\omega(m)=i\}|\leqslant \pi\ab(X)+\ab(4\ab|ab|n^2+\beta)Xe^{-\frac{2\sqrt{\log X}}{n}}
\] 
for all $X\geqslant1$ and
\begin{align*}
   \ab|\{m\leqslant X:\kl_n(a,b;m)\neq\eta_i f(m),\mu^{2}(m)=1\text{ if }\omega(m)=i\}|=\frac{X}{\zeta\ab(2)}+O(\frac{X}{logX})   
\end{align*}
as $X\rightarrow+\infty$.
Here, $\beta=8+2C$ and $C$ is the absolute constant in Lemma \ref{lm:k-prime bound}, and $\zeta(2)=\sum_{m=1}^{\infty}\frac{1}{m^2}$.
\end{theorem}
\proof
When $1\leqslant X\leqslant\ab(2\ab|ab|)^2$, by definition of $r(X)$ we get
\begin{align*}
r(X)\leqslant X\leqslant2\ab|ab|X^\frac{1}{2}\leqslant 4\ab|ab|Xe^{-\sqrt{\log X}}.
\end{align*}
When $X>(2\ab|ab|)^2$, combining Lemmas~\ref{pro:a},~\ref{pro:b},~\ref{pro:c} and~\ref{pro:d}, we obtain
\begin{align*}
&r(X)\leqslant\Sigma^1+\Sigma^3+\Sigma^3+\Sigma^4\\
&\leqslant4\ab|ab|n^2Xe^{-\frac{2\sqrt{\log X}}{n}}+\ab(\pi^0\Big(\frac{X}{2}\Big)+2X^{\frac{2}{n+1}})+\ab(X^{\frac{1}{n+1}})+
2CXe^{-\sqrt{\log X}}\\
&\leqslant \pi^0\Big(\frac{X}{2}\Big)+\ab(4\ab|ab|n^2+8+2C)Xe^{-\frac{2\sqrt{\log X}}{n}}.
\end{align*}
Here, $C$ is the absolute constant in Lemma~\ref{lm:k-prime bound}. Denoting $\beta\coloneqq8+2C$, evidently, for all $X\geqslant1$ we find
\begin{align}\label{pro:r(X)bound}
r(X)\leqslant\pi^0\Big(\frac{X}{2}\Big)+\ab(4\ab|ab|n^2+\beta)Xe^{-\frac{2\sqrt{\log X}}{n}}.
\end{align}
Thus, for all $X\geqslant1$ we have
\begin{align*}
&\ab|\{m\leqslant X:\mu^{2}(m)=1,\kl_n(a,b;m)=\eta_i f(m)\text{ if }\omega(m)=i\}|\\
&=r_1(X)+r(X)\\
&\leqslant r_1(X)+\pi^0\Big(\frac{X}{2}\Big)+\ab(4\ab|ab|n^2+\beta)Xe^{-\frac{2\sqrt{\log X}}{n}}\\
&\leqslant\pi\ab(X)+\ab(4\ab|ab|n^2+\beta)Xe^{-\frac{2\sqrt{\log X}}{n}}.
\end{align*}
From \cite[Chapter I.3, Theorem 3.10]{Tenenbaum15}, we know that
\begin{align*}
\ab|\{m\leqslant X:\mu^{2}(m)=1\}|=\sum_{d\leqslant\sqrt{X}}\mu(d)\lfloor\frac{X}{d^2}\rfloor
=\frac{X}{\zeta(2)}+O(\sqrt{X}),
\end{align*}
where $\zeta(2)=\sum_{m=1}^{\infty}\frac{1}{m^2}$. Therefore we derive that
\begin{align*}
    \ab|\{m\leqslant X:\mu^{2}(m)=1,\kl_n(a,b;m)\neq\eta_i f(m)\text{ if }\omega(m)=i\}|
    =\frac{X}{\zeta\ab(2)}+O(\frac{X}{\log X}).
\end{align*}
by the Prime Number Theory.
\endproof

In the remainder of this section, we will prove the desired upper bounds of $\eqref{sum1}$, \eqref{sum2}, \eqref{sum3} and \eqref{sum4} in the following four lemmas separately.

\begin{lemma}\label{pro:a}
For all $X\geqslant1$,
\begin{align*}
\Sigma^1=\sum_{k\in[2,\frac{\sqrt{\log X}}{3}]}\sum_{p\in(1,X^{\frac{n}{n+1}}]}\ab|R_k^p(X)|\leqslant4\ab|ab|n^2Xe^{-\frac{2\sqrt{\log X}}{n}}.
\end{align*}
\end{lemma}
\proof
When $1\leqslant X\leqslant(\ab|ab|)^{k+1}$, we have
 \begin{align*}
\sum_{p\in(1,X^{\frac{n}{n+1}}]}\ab|R_k^p(X)|&\leqslant r_k(X)\leqslant X\\
&\leqslant\ab|ab|X^{1-\frac{1}{k+1}}.
\end{align*}
When $X> (\ab|ab|)^{k+1}$, combining Lemma \ref{lm:Sigma1}, \ref{lm:Sigma2} and \ref{lm:Sigma3}, we get
\begin{align*}
\sum_{p\in(1,X^{\frac{n}{n+1}}]}\ab|R_k^p(X)|&=\Sigma^1_k+\Sigma^2_k+\Sigma^3_k\\
&\leqslant \ab(3+5\sqrt{nk})X^{1-\frac{1}{nk}}.
\end{align*}

Hence, for all $X\geqslant1$ we obtain
\begin{align*}
\sum_{k\in[2,\frac{\sqrt{\log X}}{3}]}\sum_{p\in(1,X^{\frac{n}{n+1}}]}\ab|R_k^p(X)|
&\leqslant\sum_{k\in[2,\frac{\sqrt{\log X}}{3}]}\max\{\ab|ab|X^{1-\frac{1}{k+1}},
\ab(3+5\sqrt{nk})X^{1-\frac{1}{nk}}\}\\
&\leqslant\sum_{k\in[2,\frac{\sqrt{\log X}}{3}]}5\ab(\ab|ab|+\sqrt{nk})X^{1-\frac{1}{nk}}\\
&\leqslant5(\ab|ab|+\sqrt{n})Xe^{-\frac{3}{n}\sqrt{\log X}}\ab(\frac{\sqrt{\log X}}{3})^{\frac{3}{2}}\\
&=5(\ab|ab|+\sqrt{n})Xe^{-\frac{2}{n}\sqrt{\log X}}\ab[\frac{n}{2}\ab(\frac{2\sqrt{\log X}}{3n}e^{-\frac{2\sqrt{\log X}}{3n}})]^{\frac{3}{2}}\\
&\leqslant5\ab(\ab|ab|+\sqrt{n})\ab(\frac{n}{2})^\frac{3}{2}Xe^{-\frac{2\sqrt{\log X}}{n}}\\
&\leqslant4\ab|ab|n^2Xe^{-\frac{2\sqrt{\log X}}{n}}.
\end{align*}
\endproof

\begin{lemma}\label{pro:b}
For all $X>(\ab|ab|)^{\frac{n+2}{n}}$,
\begin{align*}
\Sigma^2=\sum_{k\in[2,\frac{\sqrt{\log X}}{3}]}\sum_{\substack{p\in\mathcal{P}^0\\ p\in(X^{\frac{n}{n+1}},\frac{X}{L_k}]}}\ab|R_k^p(X)|
\leqslant \pi^0\Big(\frac{X}{2}\Big)+2X^{\frac{2}{n+1}}.
\end{align*}
\end{lemma}
\proof
Since $R^p_k(X)\subseteq\Pi_{k-1}$ and $\ab|\Pi_{k_1}\cap\Pi_{k_2}|=0$ for any $k_1\neq k_2$, we have
\begin{align*}
\sum_{k\in[2,\frac{\sqrt{\log X}}{3}]}\sum_{\substack{p\in\mathcal{P}^0\\ p\in(X^{\frac{n}{n+1}},\frac{X}{L_k}]}}\ab|R_k^p(X)|
&\leqslant \sum_{\substack{p\in\mathcal{P}^0\\ p\in(X^{\frac{n}{n+1}},\frac{X}{2}]}}\sum_{k\in[2,\frac{\sqrt{\log X}}{3}]}\ab|R_k^p(X)|\\
&=\sum_{p\in\mathcal{P}^0\cap(X^{\frac{n}{n+1}},\frac{X}{2}]}\Big|\bigcup_{k\in[2,\frac{\sqrt{\log X}}{3}]}R_k^p(X)\Big|.
\end{align*}
Denote 
\[
S^p(X)\coloneqq\bigcup_{k\in[2,\frac{\sqrt{\log X}}{3}]}R_k^p(X).
\]
Then we know that
\begin{align*}
&\sum_{k\in[2,\frac{\sqrt{\log X}}{3}]}\sum_{\substack{p\in\mathcal{P}^0\\ p\in(X^{\frac{n}{n+1}},\frac{X}{L_k}]}}\ab|R_k^p(X)|\\
&\leqslant\sum_{\substack{p\in\mathcal{P}^0\cap(X^{\frac{n}{n+1}},\frac{X}{2}]\\|S^p(X)|\leqslant1}}\ab|S^p(X)|+\sum_{\substack{p\in\mathcal{P}^0\cap(X^{\frac{n}{n+1}},\frac{X}{2}]\\|S^p(X)|\geqslant2}}\ab|S^p(X)|.
\end{align*}
Evidently,
\begin{align*}
\sum_{\substack{p\in\mathcal{P}^0\cap(X^{\frac{n}{n+1}},\frac{X}{2}]\\|S^p(X)|\leqslant1}}\ab|S^p(X)|&\leqslant \sum_{\substack{p\in\mathcal{P}^0\cap(X^{\frac{n}{n+1}},\frac{X}{2}]}}1\\
&\leqslant \pi^0\Big(\frac{X}{2}\Big).
\end{align*}
Therefore, it suffices to prove that
\begin{align*}
\sum_{\substack{p\in\mathcal{P}^0\cap(X^{\frac{n}{n+1}},\frac{X}{2}]\\|S^p(X)|\geqslant2}}\ab|S^p(X)|
\leqslant 2X^{\frac{2}{n+1}}.
\end{align*}
By \eqref{lm:p-sum}, we get
\begin{align*}
S^p(X)=\bigcup_{k\in[2,\frac{\sqrt{\log X}}{3}]}R_k^p(X)&\subseteq\bigcup_{k\in[2,\frac{\sqrt{\log X}}{3}]}\Pi_{k-1}\Big(\frac{X}{p},p\Big)\\
&\subseteq\bigcup_{k\in[2,\frac{\sqrt{\log X}}{3}]}\Pi_{k-1}\ab(X^{\frac{1}{n+1}}),
\end{align*}
for all $p\in(X^{\frac{n}{n+1}},\frac{X}{2}]$.
Denote
\[
W(X)\coloneqq\bigcup_{k\in[2,\frac{\sqrt{\log X}}{3}]}\Pi_{k-1}\ab(X^{\frac{1}{n+1}}).
\]
Then $W(X)\subseteq\{m\leqslant X^{\frac{1}{n+1}}: m\in\mathbb{N}\}$ and 
\begin{align}\label{W(X)}
\ab|W(X)|\leqslant X^{\frac{1}{n+1}}.
\end{align}
Since $X\geqslant(\ab|ab|)^{\frac{n+2}{n}}$, for any two primes $p\neq q\in(X^{\frac{n}{n+1}},\frac{X}{2}]$, from Lemma \ref{lm:R^p-intersection} we get
\begin{align}\label{2u-inter}
\ab|S^p(X)\cap S^q(X)|\leqslant1.
\end{align}
Denote 
\[
U(X)\coloneqq\{p\in\mathcal{P}^0\cap(X^{\frac{n}{n+1}},\frac{X}{2}]:\ab|S^p(X)|\geqslant2\}. 
\] 
Applying Lemma~\ref{lm:intersection}, we have
\begin{align}\label{U-W}
\sum_{\substack{p\in\mathcal{P}^0\cap(X^{\frac{n}{n+1}},\frac{X}{2}]\\|S^p(X)|\geqslant2}}\ab|S^p(X)|\leqslant\ab|U(X)|+\ab|W(X)|\sqrt{\ab|U(X)|}.
\end{align}

From \eqref{2u-inter} we know that any 2-element subset of $W(X)$ is contained in at most one
$S^p(X)$ for all primes $p\in U(X)$. By the definition of $U(X)$ and the pigeonhole principle, then we have
\begin{align}\label{sqaure}
\ab|U(X)|\leqslant\binom{\ab|W(X)|}{2}
\leqslant\ab|W(X)|^2.
\end{align}
Combining~\eqref{U-W},~\eqref{sqaure} and~\eqref{W(X)}, we obtain
\begin{align*}
\sum_{\substack{p\in\mathcal{P}^0\cap(X^{\frac{n}{n+1}},\frac{X}{2}]\\|S^p(X)|\geqslant2}}\ab|S^p(X)|&
\leqslant2\ab|W(X)|^2\\
&\leqslant 2X^{\frac{2}{n+1}}.
\end{align*}
Hence the proof is finished.
\endproof

\begin{lemma}\label{pro:c}
For all $X>(\ab|ab|)^{\frac{n+2}{n}}$,
\begin{align*}
\Sigma^3=\sum_{k\in[2,\frac{\sqrt{\log X}}{3}]}\sum_{\substack{p\in R_1\\ p\in(X^{\frac{n}{n+1}},\frac{X}{L_k}]}}\ab|R_k^p(X)|
\leqslant X^{\frac{1}{n+1}}.
\end{align*}
\end{lemma}
\proof
Since $R^p_k(X)\subseteq\Pi_{k-1}$ and $\ab|\Pi_{k_1}\cap\Pi_{k_2}|=0$ for any $k_1\neq k_2$, we have
\begin{align*}
\sum_{k\in[2,\frac{\sqrt{\log X}}{3}]}\sum_{\substack{p\in R_1\\ p\in(X^{\frac{n}{n+1}},\frac{X}{L_k}]}}\ab|R_k^p(X)|
&\leqslant \sum_{\substack{p\in R_1\\ p\in(X^{\frac{n}{n+1}},\frac{X}{2}]}}\sum_{k\in[2,\frac{\sqrt{\log X}}{3}]}\ab|R_k^p(X)|\\
&=\sum_{p\in R_1\cap(X^{\frac{n}{n+1}},\frac{X}{2}]}\Big|\bigcup_{k\in[2,\frac{\sqrt{\log X}}{3}]}R_k^p(X)\Big|.
\end{align*}
For all $p\in(X^{\frac{n}{n+1}},\frac{X}{2}]$, from \eqref{lm:p-sum} we get
\begin{align}\label{good subset}
\bigcup_{k\in[2,\frac{\sqrt{\log X}}{3}]}R_k^p(X)\subseteq\bigcup_{k\in[2,\frac{\sqrt{\log X}}{3}]}\Pi_{k-1}\Big(\frac{X}{p},p\Big)
\subseteq\bigcup_{k\in[2,\frac{\sqrt{\log X}}{3}]}\Pi_{k-1}\ab(X^{\frac{1}{n+1}}).
\end{align}
Since $X>(\ab|ab|)^{\frac{n+2}{n}}$, using Lemma \ref{lm:R^p-intersection0}, for any two different primes $p,q\in R_1\cap(X^{\frac{n}{n+1}},\frac{X}{2}]$ we have
\begin{align}\label{2u-intersection}
\Big|\bigcup_{k\in[2,\frac{\sqrt{\log X}}{3}]}R_k^p(X)\cap \bigcup_{k\in[2,\frac{\sqrt{\log X}}{3}]}R_k^q(X)\Big|=0.
\end{align}
Combining~\eqref{good subset} and~\eqref{2u-intersection}, we obtain
\begin{align*}
\sum_{k\in[2,\frac{\sqrt{\log X}}{3}]}\sum_{\substack{p\in R_1\\ p\in(X^{\frac{n}{n+1}},\frac{X}{L_k}]}}\ab|R_k^p(X)|
&\leqslant \sum_{p\in R_1\cap(X^{\frac{n}{n+1}},\frac{X}{2}]}\Big|\bigcup_{k\in[2,\frac{\sqrt{\log X}}{3}]}R_k^p(X)\Big|\\
&\leqslant\Big|\bigcup_{k\in[2,\frac{\sqrt{\log X}}{3}]}\Pi_{k-1}\ab(X^{\frac{1}{n+1}})\Big|
\leqslant\ab|\{m\leqslant X^{\frac{1}{n+1}}:m\in\mathbb{N}\}|\\
&\leqslant X^{\frac{1}{n+1}}.
\end{align*}
This completes the proof.
\endproof

\begin{lemma}\label{pro:d}
For all $X\geqslant3$, we have
\begin{align*}
\Sigma^4=\sum_{k\in(\frac{\sqrt{\log X}}{3},\frac{\log X}{\log2}]}r_k(X)
\leqslant 2CXe^{-\sqrt{\log X}}.
\end{align*}
Here $C$ is absolute constant in Lemma~\ref{lm:k-prime bound}.
\end{lemma}
\proof
Applying Lemma \ref{lm:k-prime bound}, we obtain
\begin{align*}
\sum_{k\in(\frac{\sqrt{\log X}}{3},\frac{\log X}{\log2}]}r_k(X)
&\leqslant\sum_{k\in(\frac{\sqrt{\log X}}{3},\frac{\log X}{\log2}]}\pi_k\ab(X)\\
&\leqslant\sum_{k\in(\frac{\sqrt{\log X}}{3},\frac{\log X}{\log2}]}C\frac{X}{\log X}e^{-\sqrt{\log X}}\\
&\leqslant 2CXe^{-\sqrt{\log X}}.
\end{align*}
\endproof

\section{Proof of Theorem \ref{thm:repal}}\label{Sec.5}
As in the previous section, Theorem \ref{thm:repal} follows from the next theorem by setting $\eta_i=\eta$ for all $i$.
\begin{theorem}\label{thm:repal1}
Given integers $n\geqslant2,a\neq 0,b\neq0$ and sequence of nonzero complex numbers $(\eta_i)_{i=1}^{\infty}$. Let $f$ be a complex valued multiplicative function satisfying $\kl_n(a,b;p)=\eta_1 f(p)$ for all but finitely many primes $p$. Then
\begin{align*}
\ab|\{m\in\Pi_k(X):\kl_n(a,b;m)=\eta_k f(m)\}|\leqslant\ab(\ab|ab|+nk)X^{1-\frac{1}{k+1}}+\binom{p_f}{k}.
\end{align*}
for every $X\geqslant1$ and integer $k\geqslant2$. 
Here $p_f$ is the maximal prime $p$ with $\kl_n(a,b;p)\neq\eta_1 f(p)$, $\binom{p_f}{k}$ is the $k$-th binomial coefficient.
\end{theorem}
\proof
In view of \eqref{lm:sum}, we get
\begin{align*}
r_{k}(X)& =\mathop{\sum}_{p\in \left(1,\frac{X}{L_k}\right]}\ab|R^p_{k}(X)|\\
& =\mathop{\sum}_{p\in \mathcal{P}^0\cap\left(1,\frac{X}{L_k}\right]}\ab|R^p_{k}(X)|+ \mathop{\sum}_{p\in R_1\cap\left(1,\frac{X}{L_k}\right]}|R^p_{k}(X)|.
\end{align*}
Denote $p_f$ is the maximal element of $\mathcal{P}^0$ (\eqref{P^0} the set of primes $p$ with $\kl_n(a,b;p)\neq\eta_1 f(p)$). By the definition of $R^p_k(X)$ we know that
\begin{align*}
\mathop{\sum}_{p\in \mathcal{P}^0\cap\left(1,\frac{X}{L_k}\right]}\ab|R^p_{k}(X)|\leqslant \pi_k\Big(\frac{X}{L_k},p_f\Big)\leqslant \binom{p_f}{k}.
\end{align*}
Hence, we only need to prove that
\begin{align*}
\mathop{\sum}_{p\in R_1\cap\left(1,\frac{X}{L_k}\right]}\ab|R^p_{k}(X)|\leqslant\ab(\ab|ab|+nk)X^{1-\frac{1}{k+1}}.
\end{align*}

When $X\leqslant\ab(\ab|ab|)^{k+1}$, we have
\begin{align}\label{a+b}
\mathop{\sum}_{p\in R_1\cap\left(1,\frac{X}{L_k}\right]}\ab|R^p_{k}(X)|\leqslant r_k(X)\leqslant X
\leqslant\ab|ab|X^{1-\frac{1}{k+1}}.
\end{align}

Therefore, it suffices to prove that the following inequality holds for all $X>(\ab|ab|)^{k+1}$.
\begin{align}\label{2k}
\mathop{\sum}_{p\in R_1\cap\left(1,\frac{X}{L_k}\right]}\ab|R^p_{k}(X)|
\leqslant nkX^{1-\frac{1}{k+1}}.
\end{align}
From Lemma \ref{lm:Sigma1}, we know that
\begin{align}\label{k+1}
\mathop{\sum}_{p\in R_1\cap(1,X^{\frac{1}{k+1}}]}\ab|R^p_{k}(X)|\leqslant X^{\frac{k}{k+1}}.
\end{align}

Then it suffices to prove that the following inequality holds for all $X>(\ab|ab|)^{k+1}$.
\begin{align}\label{2k-1}
\mathop{\sum}_{p\in R_1\cap(X^{\frac{1}{k+1}},\frac{X}{L_k}]}\ab|R^p_{k}(X)|
\leqslant(nk-1)X^{1-\frac{1}{k+1}}.
\end{align}

By \eqref{lm:p-sum}, for all $p\in(X^{\frac{1}{k+1}},X]$, we have
\begin{align}\label{subset0}
R_k^p(X)\subseteq\Pi_{k-1}\Big(\frac{X}{p},p\Big)\subseteq\Pi_{k-1}\ab(X^{\frac{k}{k+1}}).
\end{align}
For $X>(\ab|ab|)^{k+1}$, applying Lemma~\ref{lm:R^p-intersection0}, we obtain
\begin{align}\label{intersection0}
\ab|\bigcap_{i=1}^{nk}R^{p_i}_{k}(X)|=0
\end{align}
for any $nk$ distinct primes $p_{1},\dots,p_{nk}\in R_1\cap(X^{\tfrac{1}{k+1}},X]$.  
By \eqref{subset0} and the intersection condition \eqref{intersection0}, we know that any 1-element subset of $\Pi_{k-1}\ab(X^{\frac{k}{k+1}})$ is a subset of at most $(nk-1)$ distinct sets $R_k^p(X)$. Double counting on the collection of $1$-element subsets in all $R_k^p(X)$ then gives
\begin{align}
\mathop{\sum}_{p\in R_1\cap(X^{\frac{1}{k+1}},\frac{X}{L_k}]}\ab|R^p_{k}(X)|\leqslant(nk-1)\ab|\Pi_{k-1}\ab(X^{\frac{k}{k+1}})|
\leqslant(nk-1)X^{1-\frac{1}{k+1}}.
\end{align}
Hence the proof is finished.

\endproof

\appendix

\section{Generalization}

This section devotes to extend all our theorems of hyper-Kloosterman sums to a more general setting and provide an axiomatic treatment of these theorems. 

Denote
\begin{align*}
\mathbb{N}_s\coloneqq\{m\in\mathbb{N}:\mu^2(m)=1\}
\end{align*}
Inspired by the proof of previous theorems, we introduce the following concept.
\begin{definition}\label{def:Kloostermanian}
 For any function $\mathcal{E}(t;m)\colon \mathbb{Z}\times\mathbb{N}\rightarrow\mathbb{C}$, we call $\mathcal{E}(1;m)$ a \emph{Kloostermanian function} in $m\in\mathbb{N}$, if there exist integers $B\geqslant1,N\geqslant2$ such that the following three properties are satisfied. And we also call $(B;N)$ a constant pair of $\mathcal{E}(1;m)$.
 \hspace*{\fill}
\begin{enumerate}[leftmargin=*,align=left]
\item \emph{T-multiplicative property}:
For any $m,m_1,m_2\in\mathbb{N}_s$ with $(m_1,m_2)=1$ and $t\in\mathbb{Z}$, 
\[
\mathcal{E}(t;m)=\mathcal{E}(t+m;m);\,\,\mathcal{E}(t;m)\in \mathbb{Q}(\zeta_{m});
\]
\[\mathcal{E}(t;m_1m_2)=\mathcal{E}(t\overline{m_2};m_1)\mathcal{E}(t\overline{m_1};m_2).
\]

\item \emph{Non-vanishing property}: For any prime $p$ and $t\in\mathbb{Z}$ with $(p,t)=1$, $\mathcal{E}(t;p)\neq0$.

\item \emph{Irrational property}: For any prime $p>B$ and integers $t_1,t_2$ with $(p,t_1t_2)=1$, if 
\[
\mathcal{E}(t_1;p)/\mathcal{E}(t_2;p)\in \mathbb{Q}
\]
then $(\overline{t_1}t_2)^N\equiv1\bmod p$. 
\end{enumerate}
\end{definition}

For any integers $n\geqslant2,a\neq0,b\neq0$, define $K^{a,b}_n(t;m)\coloneqq\kl_n(at,bt;m)$. Then we know that $K^{a,b}_n(1;m)$ is a Kloostermanian function in $m$ with constant pair $(\ab|ab|;n)$ from Lemma~\ref{lm:twistmuliplicative}, Corollary~\ref{cor:nonzero} and Lemma~\ref{lm:irational}. Indeed, all of our theorems of $\kl_n(a,b;m)$ can be extend to general Kloostermanian functions, yielding the following theorem.
\begin{theorem}\label{thm:Kloostermanian}
Let $\mathcal{E}(1;m)$ be a Kloostermanian function in $m$ with constant pair $(B;N)$. Then for any sequence of nonzero complex numbers $(\eta_i)^{\infty}_{i=1}$ and complex valued multiplicative function $f$, we have
\[
     \ab|\{m\leqslant X:\mathcal{E}(1;m)=\eta_k f(m),m\in\Pi_k\cap\mathbb{N}_s\}|\leqslant \pi\ab(\frac{X}{L_k})+5\ab(B+\sqrt{Nk})X^{1-\frac{1}{Nk}}
\]
and 
\begin{align*}
 \ab|\{m\leqslant X:m\in\mathbb{N}_s, \mathcal{E}(1;m)=\eta_i f(m)\text{ if } \omega(m)=i\}|
 \leqslant \pi\ab(X)+(4BN^2+\beta)Xe^{-\frac{2\sqrt{\log X}}{N}}
\end{align*}
for all integer $k\geqslant2$ and real number $X\geqslant1$.

In particular, if $\mathcal{E}(1;p)=\eta_1f(p)$ holds for all but finite primes $p$, then we further have
\[
     \ab|\{m\leqslant X:\mathcal{E}(1;m)=\eta_k f(m),m\in\Pi_k\cap\mathbb{N}_s\}|\leqslant \ab(B+Nk)X^{1-\frac{1}{k+1}}+\binom{p_f}{k}
\]
for all integer $k\geqslant2$ and real number $X\geqslant1$.
Here, $p_f$ is the maximal prime $p$ with $\mathcal{E}(1;p)\neq\eta_1f(p)$.
\end{theorem}
\proof
Note that Theorems~\ref{thm:distinctness},~\ref{thm:square-free1} and~\ref{thm:repal1} are completely determined by Lemmas~\ref{lm:R^p-intersection} and~\ref{lm:R^p-intersection0}. And it is also easy to check that these two lemmas are completely determined by the fact that $\kl_n(a,b;m)$ is a Kloostermanian function in $m$ with constant pair $(\ab|ab|;n)$ for any given integers $n\geqslant2$, $a\neq0,b\neq0$. Therefore, the theorem follows immediately.
\endproof

\begin{remark}\label{rmk:n to p}
 If $\mathcal{E}(1;m)$ possesses T-multiplicative property but does not possess Non-vanishing property. Then there exists $p_0$ such that $\mathcal{E}(1;p_0)=0$ and $\mathcal{E}(1;mp_0)=0$ for any $m\in\mathbb{N}$ with $p\nmid m$. Therefore, for any integer $k\geqslant3$ and complex valued multiplicative function $f$ with $f(p)=\mathcal{E}(1;p)$, we have
\[
     \ab|\{m\in\Pi_k(X):\mathcal{E}(1;m)=f(m)\}|\geqslant \ab|\{m\in\Pi_{k-1}\ab( \frac{X}{p_0}): (m,p_0)=1\}|\gg \pi_{k-1}\ab(X).
\]
This means all estimates in Theorems~\ref{thm:distinctness},~\ref{thm:square-free} and~\ref{thm:repal} do not hold for $\mathcal{E}(1;m)$.
\end{remark}

In fact, our primary motivation for defining Kloostermanian function is to generalize and unify all algebraic exponential sums possessing Non-vanishing property and Irrational property, like hyper-Kloosterman sums. We now return to algebraic exponential sums.

Let $l\geqslant1,r\geqslant1$ be integers. For polynomials $F_1,\dots,F_l,g,h\in\mathbb{Z}[x_1,\dots,x_r]$, denote 
\begin{align*}
V_{F,g}(\frac{\mathbb{Z}}{m\mathbb{Z}})\coloneqq\{x\in\mathbb{A}^n(\frac{\mathbb{Z}}{m\mathbb{Z}}):F_1(x)=\cdots=F_l(x)=0,g(x)\neq0\}
\end{align*}
and
\begin{align}\label{algebraic exp sum}
S_{F,g}(h;m)\coloneqq\sum_{x\in V_{F,g}(\frac{\mathbb{Z}}{m\mathbb{Z}})}e\Big(\frac{h(x)}{m}\Big)
\end{align}    
for $m\in\mathbb{N}$. 

It is easy to see that $\kl_n(a,b;m)$ is a special case of $S_{F,g}(h;m)$. For $S_{F,g}(h;m)$, we have the following theorem:
\begin{theorem}\label{thm:exp sum}
The exponential sum $S_{F,g}(h;m)$ is a Kloostermanian function in $m$ with constant pair $(B;N)$, if there exist integers $B\geqslant1,N\geqslant2$ such that the following two conditions are satisfied.
\hspace*{\fill}
\begin{itemize}[leftmargin=*,align=left]
\item $S_{F,g}(h;p)\neq0$ holds for all primes $p$.
\item For any prime $p>B$ and integer $t$ with $(p,t)=1$, if 
\[
S_{F,g}(h;p)/S_{F,g}(th;p)\in \mathbb{Q}
\]
then $t^N\equiv1\bmod p$. 
\end{itemize}
\end{theorem}
\proof
Evidently, $S_{F,g}(h;m)$ has T-multiplicative property. It is also easy to see that the above two conditions are equivalent to $S_{F,g}(h;m)$ having Non-vanishing property and Irrational property.
\endproof
\begin{remark}\label{re:exp sum}
It is rather difficult to provide a unified definition for all algebraic exponential sums. Therefore, we take 
$S_{F,g}(h;m)$ as a representative example to illustrate that any algebraic exponential sum which satisfies the conditions of Theorem \ref{thm:exp sum} is a Kloostermanian function.
\end{remark}

In the following theorem we obtain necessary and sufficient criterion for the coefficients $a,b$ in the Birch sum $B(a,b;m)$, another important exponential sum in the class $S_{F,g}$, for $B(a,b;m)$ to satisfy the relevant conditions as in Theorem~\ref{thm:exp sum} and hence be Kloostermanian.
\begin{theorem}\label{thm:Birch sum}
The Birch sum:
\begin{align*}
B(a,b;m)=\sum_{x \bmod{m}}e\Big(\frac{ax^3+bx}{m}\Big)
\end{align*}
is a Kloostermanian function in $m$ if and only if $a,b$ satisfy the following two conditions.
\hspace*{\fill}
\begin{itemize}[leftmargin=*,align=left]
\item  $a\neq0$.
\item  For any prime $p$, if $p=3$ then $p|(a+b)$; if $p|a$ then $p|b$; if $p\equiv2\bmod{3}$ with $p|b$ then $p|a$.
\end{itemize}
In particular, $(|ab|;2)$ is aconstant pair of $B(a,b;m)$ if both these conditions are satisfied.
\end{theorem}
\proof
If $a=0$, then $B(a,b;m)$ is a multiplicative function and evidently not a Kloostermanian function. 
If $a\neq0$ but the second condition not satisfied is not satisfied, then $B(a,b;m)$ does not possess Non-vanishing property.
Thus, the necessity is proved; it remains only to prove the sufficiency.

Suppose that integers $a,b$ satisfy these two conditions. It is easy to check that $ab\neq0$.
By Theorem~\ref{thm:exp sum}, it suffices to prove that $B(a,b;m)$ satisfies the following two conditions:
\hspace*{\fill}
\begin{itemize}[leftmargin=*,align=left]
\item $B(a,b;p)\neq0$ holds for all primes $p$.
\item  For any prime $p>\ab|ab|$ and integers $t$ with $(p,t)=1$, if
$B(a,b;p)/B(at,bt;p)\in\mathbb{Q}$ then $t^2\equiv1\bmod p$.    
\end{itemize}
Note that $B(ta,tb,5)=\kl_2(tb,ta;5)$ for all $t\in\mathbb{Z}$, thus it is evident for $p\leqslant5$.

Henceforth, we assume that $p\geqslant7$. If $p|a$ and $p|b$, obviously $B(a,b;p)=p\neq0$.
And if $(p,ab)=1$ or $p\equiv1\bmod{3}$ with $p|b$ and $(p,a)=1$, then by \cite[Lemma 2.2]{Livne87} we have
\begin{align}\label{congruence1}
B(at,bt;p)\equiv-\lambda^l(at)^l\bmod {\lambda^{l+1}}\text{ if } p\equiv1\bmod{3}
\end{align}
and
\begin{align}\label{congruence2}
B(at,bt;p)\equiv-l\lambda^l(at)^{l-1}\bmod {\lambda^{l+1}} \text{ if } p\equiv2\bmod{3}
\end{align}
for any integer $t$ with $(p,t)=1$. Here, $l=\lfloor\frac{p+1}{3}\rfloor$ and $\lambda=\zeta_p-1$.
Hence, $B(a,b;p)\neq0$.

Assume that $p>\ab|ab|$ and $B(a,b;p)/B(at,bt;p)\in\mathbb{Q}$ for integer $t$ with $(p,t)=1$. Since $B(at,bt;p)$ is a Galois conjugate of $B(a,b;p)$, we have 
\[
B(a,b;p)=\pm B(at,bt;p).
\]
If $B(a,b;p)=-B(at,bt;p)$, from~\eqref{congruence1} and~\eqref{congruence2} we obtain
\begin{align*}
t^{\frac{p-1}{3}}=t^l\equiv-1\bmod{p} \text{ if } p\equiv1\bmod{3}
\end{align*}
and
\begin{align*}
t^{\frac{p-2}{3}}=t^{l-1}\equiv-1\bmod{p} \text{ if } p\equiv2\bmod{3}.
\end{align*}
Therefore, we conclude that $p\equiv2\bmod{3}$ and $t\equiv-1\bmod{p}$.

From \cite[Corollary 2.3]{Livne87} we know that $B(a,b;p)$ is a generator of the maximal real subfield
$\mathbb{Q}(\zeta_p)^{+}$ of $\mathbb{Q}(\zeta_p)$. Thus, if $B(a,b;p)=B(at,bt;p)$ then we have $t\equiv\pm1\bmod{p}$. This completes the proof.
\endproof

From Theorem~\ref{thm:Birch sum}, we know that there are only a zero-density of integer pairs $a,b$ such that $B(a,b;m)$ to be a Kloostermanian function in $m$. This leads us to introduce the following concepts.

For any positive density subset $\mathcal{M}$ of $\mathcal{P}$, denote
\begin{align*}
\mathbb{N}(\mathcal{M})\coloneqq\{n\in\mathbb{N}:\mu^2(n)=1,\text{ if }p\in\mathcal{P}\text{ and }p|n\text{ then }p\in\mathcal{M}\}
\end{align*}
and
\begin{align*}
\pi\ab(X;\mathcal{M})\coloneqq\{p\leqslant X:p\in\mathcal{M}\}.
\end{align*} 

\begin{definition}\label{def:almost-Kloostermanian}
 Let $\mathcal{E}(t;m)\colon \mathbb{Z}\times\mathbb{N}\rightarrow\mathbb{C}$ be a function possesses the following three properties for some integers $B\geqslant1,N\geqslant2$ and positive density subset $\mathcal{M}$ of $\mathcal{P}$. Then we call $\mathcal{E}(1;m)$ a 
 \emph{almost-Kloostermanian function} in $m$, and also call $(B;N)$ a constant pair of $\mathcal{E}(1;m)$, and $\mathcal{M}$ a \emph{good prime set}.
\hspace*{\fill}
\begin{enumerate}[leftmargin=*,align=left]
\item \emph{Almost-T property}:
For any $m,m_1,m_2\in\mathbb{N}(\mathcal{M})$ and $t\in\mathbb{Z}$ with $(m_1,m_2)=1$, 
\[
\mathcal{E}(t;m)=\mathcal{E}(t+m;m);\,\,\mathcal{E}(t;m)\in \mathbb{Q}(\zeta_{m});
\]
\[\mathcal{E}(t;m_1m_2)=\mathcal{E}(t\overline{m_2};m_1)\mathcal{E}(t\overline{m_1};m_2).
\]

\item \emph{Almost-N property}: For any prime $p\in\mathcal{M}$ and $t\in\mathbb{Z}$ with $(p,t)=1$, $\mathcal{E}(t;p)\neq0$.

\item \emph{Almost-I property}: For any $p\in\mathcal{M}$ with $p>B$ and integers $t_1,t_2$ with $(p,t_1t_2)=1$, if 
\[
\mathcal{E}(t_1;p)/\mathcal{E}(t_2;p)\in \mathbb{Q}
\]
then $(\overline{t_1}t_2)^N\equiv1\bmod p$. 
\end{enumerate}

 \end{definition}

As a direct generalization of Theorem~\ref{thm:Kloostermanian}, we have the following theorem.

\begin{theorem}\label{thm:almost-Kloostermanian}
Let $\mathcal{E}(1;m)$ be a almost-Kloostermanian function with constant pair $(B;N)$ and good prime set $\mathcal{M}$. Then for any sequence of nonzero complex numbers $(\eta_i)^{\infty}_{i=1}$ and complex valued multiplicative function $f\colon\mathbb{N} \rightarrow \mathbb{C}$, we have
\[
     \ab|\{m\leqslant X:\mathcal{E}(1;m)=\eta_k f(m),m\in\Pi_k\cap\mathbb{N}(\mathcal{M})\}|\leqslant \pi\Big(\frac{X}{L_k};\mathcal{M}\Big)+5\ab(B+\sqrt{Nk})X^{1-\frac{1}{Nk}}
\]
and 
\begin{align*}
 &\ab|\{m\leqslant X:m\in\mathbb{N}(\mathcal{M}), \mathcal{E}(1;m)=\eta_i f(m)\text{ if } \omega(m)=i\}|\\
 &\leqslant \pi\ab(X;\mathcal{M})+(4BN^2+\beta)Xe^{-\frac{2\sqrt{\log X}}{N}}
\end{align*}
for any integer $k\geqslant2$ and real number $X\geqslant1$.

In particular, if $\mathcal{E}(1;p)=\eta_1f(p)$ holds for all but finite $p\in\mathcal{M}$, then we further have
\[
     \ab|\{m\leqslant X:\mathcal{E}(1;m)=\eta_k f(m),m\in\Pi_k\cap\mathbb{N}_s(\mathcal{M})\}|\leqslant \ab(B+Nk)X^{1-\frac{1}{k+1}}+\binom{p_f}{k}
\]
Here, $p_f$ is the maximal $p\in\mathcal{M}$ with $\mathcal{E}(1;p)\neq\eta_1f(p)$.
\end{theorem}
\proof
The proofs of these statements is same as the proof of Theorem~\ref{thm:Kloostermanian}. All we need to do is replace $\mathcal{P}$ and $\mathbb{N}_s$ with $\mathcal{M}$ and $\mathbb{N}(\mathcal{M})$ respectively in the proof of Theorem~\ref{thm:Kloostermanian}.
\endproof

Similarly, as an example, we study the Birch sum and the Sali\'e sum and present the following two propositions.
\begin{proposition}\label{thm:Birch sum good}
The Birch sum $B(a,b;m)$ is a almost-Kloostermanian function in $m\in\mathbb{N}$ if and only if $ab\neq0$. 

In particular, if $ab\neq0$, then $(\ab|ab|;2)$ is a constant pair of $B(a,b;m)$, and the a positive density subset $\mathcal{M} $ of $\mathcal{P}$ is a good prime set if and only if it is contained in the following set:
\[
\mathcal{P}\backslash\{p\in\mathcal{P}: p=3 \text{ if }(3,a+b)=1;\,p|a\text{ and }(p,b)=1;\, p\equiv2\bmod{3}\text{ with } p|b,(p,a)=1\}. 
\]
\end{proposition}
\proof
If $a=0$, then $B(a,b;m)$ is a multiplicative function and not a almost-Kloostermanian function.

If $a\neq0=b$, then from the congruence property \eqref{congruence1}, we know that the set of primes $p$ with $B(a,b;p)\neq0$ is 
\[
\{p\in\mathcal{P}: p\equiv1\bmod{3}\text{ or }p|a\}. 
\]
In this case, we also know that for any prime $p$ with $p\equiv1\bmod{3}$ and $p>\ab|a|$, and any integer $t$ with $t^{\frac{p-1}{3}}\equiv1\bmod{p}$, $B(a,b;p)/B(at,b;p)=1$. Therefore, B(a,b;m) does not possess Almost-I property.

Suppose that $ab\neq0$. It is easy to check that the set of primes $p$ with $B(a,b;p)\neq0$ is 
\[
\mathcal{P}\backslash\{p\in\mathcal{P}: p=3 \text{ if }(3,a+b)=1;\,p|a\text{ and }(p,b)=1;\, p\equiv2\bmod{3}\text{ with } p|b,(p,a)=1\}. 
\]
From the proof of Theorem~\ref{thm:Birch sum} we know that for any prime $p>\ab|ab|$ and integer $t$ with $(p,t)=1$, if $B(a,b;p)/B(at,bt;p)\in\mathbb{Q}$
then $t^2\equiv1\bmod{p}$. Hence the proof is finished.
\endproof

\begin{proposition}\label{pro:Salie sum good}
The Sali\'e sum: 
\begin{align*}
\widetilde{S}(a,b;m)=\sum_{\substack{x \bmod{m}\\x\overline{x}\equiv1\bmod {m}}}\left(\frac{x}{m}\right)e\Big(\frac{ax+b\overline{x}}{m}\Big)
\end{align*}
is a almost-Kloostermanian function in $m$ if and only if $ab\neq0$. 
And if $ab\neq0$, then $(\ab|ab|,2)$ is a constant pair of $\widetilde{S}(a,b;m)$, and
a positive density subset $\mathcal{M} $ of $\mathcal{P}$ is a good prime set if and only if it is contained in the following set:
\[
\{p\in\mathcal{P}: \left(\frac{ab}{p} \right)=1\text{ or }0\}. 
\]
In particular, $\widetilde{S}(a,b;m)$ is a Kloostermanian function if and only if $ab=d^2$ for some integer $d\in\mathbb{N}$. 
\end{proposition}
\proof
By the definition, we know that $\widetilde{S}(a,b;m)$ possesses T-multiplicative property for all integers $a,b$. 
If $ab=0$, then $\widetilde{S}(a,b;m)$ is a multiplicative function of $m$ or a quadratic Gauss sum, evidently it not a almost-Kloostermanian function.

Assuming $ab\neq0$, by the classical theorem of Sali\'e \cite[Chapter 4, Lemma 4.4]{Sarnak90}, we get
\begin{align*}
\widetilde{S}(a,b;p)=0\text{ if }\ab(\frac{ab}{p})=-1;\,\,\widetilde{S}(a,b;p)\neq0\text{ if }\ab(\frac{ab}{p})=0.
\end{align*}
And
\begin{align*}
\widetilde{S}(a,b;p)=2\cos(\frac{4\pi x}{p})\sum_{y\bmod{p}}e(\frac{ay^2}{p}) \text{ and } x^2\equiv ab\bmod{p}\text{ if }\ab(\frac{ab}{p})=1.
\end{align*}

Hence, for any prime $p$ and integer $t$ with $(p,t)=1$, $\widetilde{S}(at,bt;p)\neq0$ if and only if $\left(ab/p\right)=1\text{ or }0$.
Furthermore, if $\left(ab/p\right)=1$ and $\widetilde{S}(a,b;p)/\widetilde{S}(at,bt;p)\in\mathbb{Q}$, then
\begin{align*}
\frac{\cos(\frac{4\pi x}{p})}{\cos(\frac{4\pi tx}{p})}\ab(\frac{t}{p})\in\mathbb{Q}\text{ and }x^2\equiv ab\bmod{p}.
\end{align*}
Therefore, we obtain $t^2\equiv1\bmod{p}$. This completes the proof of the first two statements.

By the first two statements and the Chebotarev density theorem, we know that $\widetilde{S}(a,b;m)$ is a Kloostermanian function of $m\in\mathbb{N}$ if and only if $ab=d^2$ for some $d\in\mathbb{N}$. 
Hence the proof is finished.
\endproof

\bibliographystyle{alpha}
\bibliography{ref}

\begin{thebibliography}{DKL12}

\bibitem[Del77]{Deligne77}
P.~Deligne.
\newblock Applications de la formule des traces aux sommes trigonom\'etriques.
\newblock In {\em Cohomologie \'etale}, volume 569 of {\em Lecture Notes in Math.}, pages 168--232. Springer, Berlin, 1977.

\bibitem[DKL12]{De-Luca12}
Jean-Marie De~Koninck and Florian Luca.
\newblock {\em Analytic number theory}, volume 134 of {\em Graduate Studies in Mathematics}.
\newblock American Mathematical Society, Providence, RI, 2012.
\newblock Exploring the anatomy of integers.

\bibitem[DM19]{Drappeau-Maynar19}
Sary Drappeau and James Maynard.
\newblock Sign changes of {K}loosterman sums and exceptional characters.
\newblock {\em Proc. Amer. Math. Soc.}, 147(1):61--75, 2019.

\bibitem[FM03]{Fouvry-Micheal03}
E.~Fouvry and Ph. Michel.
\newblock Crible asymptotique et sommes de {K}loosterman.
\newblock In {\em Proceedings of the {S}ession in {A}nalytic {N}umber {T}heory and {D}iophantine {E}quations}, volume 360 of {\em Bonner Math. Schriften}, page~27. Univ. Bonn, Bonn, 2003.

\bibitem[FM07]{Fouvry-Michel07}
\'E. Fouvry and Ph. Michel.
\newblock Sur le changement de signe des sommes de {K}loosterman.
\newblock {\em Ann. of Math. (2)}, 165(3):675--715, 2007.

\bibitem[HR00]{Hardy-Ramanujan17}
G.~H. Hardy and S.~Ramanujan.
\newblock The normal number of prime factors of a number {$n$} [{Q}uart. {J}. {M}ath. {\bf 48} (1917), 76--92].
\newblock In {\em Collected papers of {S}rinivasa {R}amanujan}, pages 262--275. AMS Chelsea Publ., Providence, RI, 2000.

\bibitem[Kat80]{katz80}
Nicholas~M. Katz.
\newblock {\em Sommes exponentielles}, volume~79 of {\em Ast\'erisque}.
\newblock Soci\'et\'e{} Math\'ematique de France, Paris, 1980.
\newblock Course taught at the University of Paris, Orsay, Fall 1979, With a preface by Luc Illusie, Notes written by G\'erard Laumon, With an English summary.

\bibitem[Liv87]{Livne87}
Ron Livn\'e.
\newblock The average distribution of cubic exponential sums.
\newblock {\em J. Reine Angew. Math.}, 375/376:362--379, 1987.

\bibitem[Mat11]{Matomaki11}
Kaisa Matom\"aki.
\newblock A note on signs of {K}loosterman sums.
\newblock {\em Bull. Soc. Math. France}, 139(3):287--295, 2011.

\bibitem[Mic95]{Michel95}
Ph. Michel.
\newblock Autour de la conjecture de {S}ato-{T}ate pour les sommes de {K}loosterman. {I}.
\newblock {\em Invent. Math.}, 121(1):61--78, 1995.

\bibitem[Sar90]{Sarnak90}
Peter Sarnak.
\newblock {\em Some applications of modular forms}, volume~99 of {\em Cambridge Tracts in Mathematics}.
\newblock Cambridge University Press, Cambridge, 1990.

\bibitem[SF09]{Sivak-Fischler09}
Jimena Sivak-Fischler.
\newblock Crible asymptotique et sommes de {K}loosterman.
\newblock {\em Bull. Soc. Math. France}, 137(1):1--62, 2009.

\bibitem[Ten15]{Tenenbaum15}
G\'erald Tenenbaum.
\newblock {\em Introduction to analytic and probabilistic number theory}, volume 163 of {\em Graduate Studies in Mathematics}.
\newblock American Mathematical Society, Providence, RI, third edition, 2015.

\bibitem[Wan95]{Wan95}
Da~Qing Wan.
\newblock Minimal polynomials and distinctness of {K}loosterman sums.
\newblock {\em Finite Fields Appl.}, 1(2):189--203, 1995.
\newblock Special issue dedicated to Leonard Carlitz.

\bibitem[Was97]{Washington97}
Lawrence~C. Washington.
\newblock {\em Introduction to cyclotomic fields}, volume~83 of {\em Graduate Texts in Mathematics}.
\newblock Springer-Verlag, New York, second edition, 1997.

\bibitem[Xi15]{Xi15}
Ping Xi.
\newblock Sign changes of {K}loosterman sums with almost prime moduli.
\newblock {\em Monatsh. Math.}, 177(1):141--163, 2015.

\bibitem[Xi18]{Xi18}
Ping Xi.
\newblock Sign changes of {K}loosterman sums with almost prime moduli. {II}.
\newblock {\em Int. Math. Res. Not. IMRN}, 2018(4):1200--1227, 2018.

\bibitem[Xi20]{Xi20}
Ping Xi.
\newblock When {K}loosterman sums meet {H}ecke eigenvalues.
\newblock {\em Invent. Math.}, 220(1):61--127, 2020.

\end{thebibliography}

\bigskip

\end{document}